\documentclass[11pt]{amsart}
\usepackage{amssymb
,amsthm
,amsmath
,amscd
,mathtools,
amsfonts, amssymb, mathrsfs, 
,url
,accents, hyphenat, calc, stmaryrd, mathdots,array, avant
}
\usepackage[all]{xy}
\usepackage[top=35truemm,bottom=35truemm,left=30truemm,right=30truemm]{geometry}
\usepackage[usenames]{color}

\usepackage[all]{xy}
\usepackage{rotating}

\usepackage[T3,T1]{fontenc}

\newcommand{\Z}{\mathbb{Z}}
\newcommand{\R}{\mathbb{R}}
\newcommand{\C}{\mathbb{C}}

\renewcommand{\AA}{\mathcal{A}}
\newcommand{\Sc}{\mathcal{S}}

\newcommand{\GL}{\mathrm{GL}}
\newcommand{\SL}{\mathrm{SL}}
\newcommand{\SO}{\mathrm{SO}}
\newcommand{\Sp}{\mathrm{Sp}}

\newcommand{\Rep}{\mathrm{Rep}}
\newcommand{\Irr}{\mathrm{Irr}}
\newcommand{\unit}{\mathrm{unit}}
\newcommand{\gp}{\mathrm{gp}}
\newcommand{\temp}{\mathrm{temp}}
\newcommand{\Ind}{\mathrm{Ind}}
\newcommand{\Jac}{\mathrm{Jac}}
\newcommand{\Frob}{\mathrm{Frob}}
\newcommand{\scusp}{\mathrm{scusp}}

\newcommand{\soc}{\mathrm{soc}}
\newcommand{\MVW}{\mathrm{MVW}}
\newcommand{\good}{\mathrm{good}}
\newcommand{\bad}{\mathrm{bad}}
\newcommand{\ugly}{\mathrm{ugly}}

\newcommand{\Cusp}{\mathscr{C}}
\newcommand{\Rr}{\mathscr{R}}

\newcommand{\iif}{&\quad&\text{if }}
\newcommand{\other}{&\quad&\text{otherwise}}
\newcommand{\resp}{resp.~}
\renewcommand{\1}{\mathbf{1}}
\newcommand{\rel}{\rightsquigarrow}

\newcommand{\pair}[1]{\left\langle #1 \right\rangle}
\newcommand{\half}[1]{\frac{#1}{2}}

\newtheorem{thm}{Theorem}[section]
\newtheorem{lem}[thm]{Lemma}
\newtheorem{prop}[thm]{Proposition}
\newtheorem{cor}[thm]{Corollary}
\newtheorem{rem}[thm]{Remark}
\newtheorem{defi}[thm]{Definition}

\newtheorem{alg}[thm]{Algorithm}

\numberwithin{equation}{section}

\title[The explicit Zelevinsky--Aubert duality]
{The explicit Zelevinsky--Aubert duality}
\author{Hiraku Atobe \and Alberto M{\'i}nguez}
\date{}
\subjclass[2010]{Primary 22E50; Secondary 11S37}
\keywords{Zelevinsky--Aubert duality; Derivatives; Socles}
\address{
Department of Mathematics, Hokkaido University,
Kita 10, Nishi 8, Kita-Ku, Sapporo, Hokkaido, 060-0810, Japan 
}
\email{
atobe@math.sci.hokudai.ac.jp
}
\address{
University of Vienna, 
Fakult{\"a}t f{\"u}r Mathematik,
Oskar-Morgenstern-Platz 1,
1090 Wien
}
\email{
alberto.minguez@univie.ac.at
}

\allowdisplaybreaks
\begin{document}
\maketitle

\begin{abstract}
In this paper, we give an explicit computable algorithm for the Zelevinsky--Aubert duals 
of irreducible representations of $p$-adic symplectic and odd special orthogonal groups. 
To do this, we establish explicit formulas for certain derivatives and socles. 
We also give a combinatorial criterion for the irreducibility of certain parabolically induced representations. 
\end{abstract}

\tableofcontents

\section{Introduction}
Let $F$ be a local non-Archimedean field. 
In 1980, A. Zelevinsky \cite{Z} defined an involution $\tau \mapsto \hat\tau$
on the Grothendieck group of finite length smooth representations of $\GL_n(F)$ and conjectured that this involution should preserve irreducibility.
Assuming this conjecture, in 1986, M{\oe}glin--Waldspurger \cite{MW2} studied the involution 
and gave an algorithm for computing the Langlands (or Zelevinsky) data of $\hat\tau$, for every irreducible representation $\tau$ of $\GL_n(F)$. 
Later, another explicit formula was given by Knight--Zelevinsky \cite{KZ}.
\vskip 10pt

Motivated by the Alvis--Curtis duality for finite groups \cite{A1, A2, Curtis}, 
S.-I. Kato \cite{Kato} defined an involution on the Grothendieck group of smooth finite length Iwahori-fixed representations of a split reductive group over $F$. 
In 1996, A.-M. Aubert showed that 
Kato's involution could be extended to the category of finite length smooth representations of any reductive group $G$ 
and proved that it indeed preserves irreducibility. 
Besides, using different approaches, 
Schneider--Stuhler \cite{SS} and Bernstein--Bezrukavnikov--Kazhdan \cite{Ber, BBK, Bez} 
have defined involutions on the category of smooth representations of $G$. 
For irreducible representations of $\GL_n(F)$, 
all these involutions coincide (up to the contragredient and up to a sign) with the involution defined by Zelevinsky.
\vskip 10pt

For simplicity, when restricted to the set of irreducible smooth representations of a reductive group $G$, 
this involution is commonly known as the \emph{Zelevinsky--Aubert duality} 
and it is the main topic of this article. 
This duality has many interesting applications to Koszul duality (see \cite{MR}) 
or to the Langlands program (see for example \cite{T2} or \cite{W}). 
One important property of the Zelevinsky--Aubert duality is that it does not preserve the fact of being tempered. 
In fact, in the proof of Arthur's local classification, 
the first step beyond tempered representations is to consider the Zelevinsky--Aubert dual of tempered representations \cite[\textsection 7]{Ar}. 
However, one expects that the duality preserves unitarity 
so it should be an important tool for determining the unitary dual of classical groups \cite{T1}. 
\vskip 10pt

Our goal is to extend the result of M{\oe}glin--Waldspurger to the Zelevinsky--Aubert duality, 
that is we give an algorithm for computing the Langlands data of $\hat\pi$ in terms of those of $\pi$, for every irreducible representation $\pi$ of $G$. 
As we will use the endoscopic classification of Arthur \cite{Ar} 
and M{\oe}glin's construction of the local packets \cite{Moe3}, 
we will focus on the case where $F$ is a local non-Archimedean field of characteristic $0$ 
and $G$ is either a symplectic or an odd special orthogonal group.
\vskip 10pt

There have been several attempts to explicitly describe the Zelevinsky--Aubert duality. 
There are some partial results by M{\oe}glin \cite{Moe4}, Mati{\'c} \cite{Ma1,Ma2}, 
Jantzen \cite{J-dual} and the first-named  author \cite{At2}. 
In order to explain the novelties of the present article, let us introduce some notation.
\vskip 10pt

Let $G$ be a connected algebraic reductive group defined over $F$.
Fix a minimal parabolic subgroup $P_0$ of $G$. 
We denote by $\Ind_P^G$ the normalized parabolic induction 
and by $\Jac_P$ its left adjoint functor, the Jacquet functor.
\par

Let $\Pi$ be a smooth finite length representation of $G$. 
We consider the virtual semisimple representation
\[
D_{G}(\Pi) = \sum_{P}(-1)^{\dim A_M} \left[\Ind_P^{G}(\Jac_{P}(\Pi))\right], 
\]
where $P = MN$ runs over all standard parabolic subgroups of $G$, 
and $A_M$ is the maximal split torus of the center of $M$.
Then Aubert \cite{Au} showed that 
if $\pi$ is irreducible, 
then there exists a sign $\epsilon \in \{\pm1\}$ 
such that $\hat{\pi} = \epsilon \cdot D_{G}(\pi)$ is also an irreducible representation. 
We call the map $\pi \mapsto \hat\pi$ the \emph{Zelevinsky--Aubert duality}. 
\par

It satisfies the following important properties: 
\begin{enumerate}
\item
The  dual of $\hat\pi$ is equal to $\pi$, i.e., the map $\pi \mapsto \hat\pi$ is an involution.
\item
If $\pi$ is supercuspidal, then $\hat\pi = \pi$.
\item The duality commutes with Jacquet functors (see \eqref{eq:Jacquet}).
\end{enumerate}
\vskip 10pt

Let us now restrict ourselves to the case where $G=G_n$ is 
either the split special orthogonal group $\SO_{2n+1}(F)$ or the symplectic group $\Sp_{2n}(F)$ of rank $n$. 
In this case, 
when $\pi$ (\resp $\tau_i$) is a smooth representation of $G_{n_0}$ (\resp $\GL_{d_i}(F)$), 
with $d_1+\cdots+d_r+n_0=n$,
we denote by
\[
\tau_1 \times \dots \times \tau_r \rtimes \pi 
\]
the normalized parabolically induced representation of $\tau_1 \boxtimes \dots \boxtimes \tau_r \boxtimes \pi$ 
from the standard parabolic subgroup $P$ of $G_n$ 
with Levi subgroup isomorphic to $\GL_{d_1}(F) \times \dots \times \GL_{d_r}(F) \times G_{n_0}$. 
\par

Given $\pi$ an irreducible representation of $G_n$ and a supercuspidal \emph{non self-dual} representation $\rho$ of $\GL_d(F)$ there exists a unique $k\geq 0$ and a unique irreducible representation $\pi_0$ of $G_{n_0}$, with $n=dk+n_0$ such that:
\begin{itemize}
\item $\pi$ is a unique irreducible subrepresentation of 
\begin{equation}\label{eq:subr}
\underbrace{\rho \times\dots\times \rho}_{k \text{ times}}\rtimes \pi_0.
\end{equation}
\item $k$ is maximal, in the sense that for every irreducible representation $\pi'_0$ of $G_{n_0-d}$, $\pi_0$ is not a subrepresentation of $\rho \rtimes \pi'_0$.
\end{itemize}
We call $\pi_0$ the \emph{highest $\rho$-derivative} of $\pi$ and denote it by $D_\rho^{\max}(\pi)$. An important consequence of the commutativity of the Zelevinsky--Aubert duality with Jacquet functors is that 
\begin{equation}\label{eq:derAu}
D_{\rho}^{\max}(\pi){\widehat{\;}} = D_{\rho^\vee}^{\max}(\hat\pi),
\end{equation}
where $\rho^\vee$ denotes the contragredient of $\rho$.

\vskip 15pt
We can now describe the main idea of the algorithm for explicating the Zelevinsky--Aubert dual of an irreducible representation $\pi$ of $G_n$. 
It is a two-step procedure:
\vskip 10pt

\begin{description}
\item[Step 1]
If there exists a supercuspidal non self-dual representation $\rho$ of $\GL_d(F)$ 
such that $D_\rho^{\max}(\pi) \neq \pi$, 
then we give an explicit formula of the Langlands data of $D_\rho^{\max}(\pi)$ in terms of those of $\pi$. 
By induction we can compute the Langlands data of $D_{\rho}^{\max}(\pi){\widehat{\;}} $. 
We finally give an explicit formula of the Langlands data of $\hat\pi$ 
in terms of those of $D_{\rho}^{\max}(\pi){\widehat{\;}} = D_{\rho^\vee}^{\max}(\hat\pi)$.

\item[Step 2] 
Assume finally that for all supercuspidal representation $\rho$ of $\GL_d(F)$, 
such that 
$\pi$ is a subrepresentation of $\rho \times \pi_0$ for some irreducible representation $\pi_0$ of $G_{n-d}$, 
we have that $\rho$ is self-dual. 
Then:
\begin{itemize}
\item If $\pi$ is tempered, 
then $\pi$ is ``almost supercuspidal'', 
and we can compute its Zelevinsky--Aubert dual explicitly
(see \S \ref{sec.temp}, especially Proposition \ref{temp}).

\item 
If $\pi$ is not tempered, 
then we show that there exists a supercuspidal self-dual representation $\rho$ of $\GL_d(F)$ 
such that $\pi$ is a unique irreducible subrepresentation of
\[
\underbrace{\Delta_\rho[0,-1] \times\dots\times \Delta_\rho[0,-1]}_{k \text{ times}}\rtimes \pi_0,
\]
for some irreducible representation  $\pi_0$ of $G_{n_0}$, with $n=2dk+n_0$, 
and some positive integer $k$ maximal as above, 
where $\Delta_\rho[0,-1]$ is a Steinberg representation (see Paragraph \ref{rep0} for a precise definition). 
We call $\pi_0$ the \emph{highest $\Delta_\rho[0,-1]$-derivative}, 
and denote it $D_{\Delta_\rho[0,-1]}^{\max}(\pi)$. 
Similar to \eqref{eq:derAu}, this derivative satisfies a formula 
\[
D_{\Delta_\rho[0,-1]}^{\max}(\pi){\widehat{\;}} = D_{Z_\rho[0,1]}^{\max}(\hat\pi), 
\]
where $D_{Z_\rho[0,1]}^{\max}(\hat\pi)$ is the \emph{highest $Z_\rho[0,1]$-derivative} of $\hat\pi$ 
(see Paragraph \ref{other-derivatives}). 
As in \emph{Step 1}, this allows us to compute by induction the Zelevinsky--Aubert dual of $\pi$. 
The precise algorithm is explained in Section \ref{sec:algo}. 
\end{itemize} 
\end{description}

Let us first  give a remark on the self-duality condition on $\rho$.
When $\rho$ is self-dual, a representation of the form \eqref{eq:subr} may have several irreducible subrepresentations and there is no simple way of distinguishing them. The same problem was already observed by Jantzen \cite{J-dual}. 
For these reasons he just considered what is called the half-integral case. 
\par

This also explains one of the differences between the case of $\GL_n(F)$ and the case of classical groups that we treat in this article. In the former case, induced representations of the form $\rho \times \pi_0$, with $\rho$ supercuspidal, have always a unique irreducible subrepresentation. The second difference is that for  $\GL_n(F)$, it is much easier to explicate the Langlands data of this subrepresentation in terms of those of $\pi$. However, the most intricate part of this article is  to explicitly describe, in terms of Langlands data, the correspondence  $\pi \leftrightarrow D_{\tau}^{\max}(\pi)$ for $\tau$ either supercuspidal non self-dual or of the form $Z_\rho[0,1]$, 
see Theorems \ref{soc-positive},  \ref{bad-soc-positive} and \ref{der-01}. 
To explicate these formulas we use matching functions as in \cite{LM} 
and $A$-parameters.  These results are interesting on their own. 
In particular, we get a combinatorial criterion for the irreducibility of parabolically induced representations 
of the form $\rho \rtimes \pi_0$ with $\rho$ non self-dual supercuspidal and $\pi_0$ irreducible, see Corollary \ref{irr_cusp}. 
We expect that the explicit formulas established in this paper 
will make M{\oe}glin's construction of $A$-packets more computable.
\vskip 10pt

This paper is organized as follows. 
In Section \ref{sec:Not}, we recall some general results on representation theory of $p$-adic classical groups. 
In Section  \ref{sec:deriv}, we define $\rho$-derivatives and other derivatives, 
and we prove some general result about them, in particular their compatibility with the Zelevinsky--Aubert duality. 
In Section \ref{sec:algo} we give our algorithm to compute the Zelevinsky--Aubert dual using derivatives and socles. 
We will prove explicit formulas for these derivatives and socles in several situations 
in Sections \ref{sub:best}, \ref{sec:pos} and \ref{sec01}. 
To do this, we review Arthur's theory of endoscopic classification in Section \ref{endoscopy} 
and the theory of matching functions at the beginning of Section \ref{sub:best}.

\subsection*{Acknowledgement}
We would like to thank Erez Lapid and Colette M{\oe}glin for useful discussions.
The first-named author was supported by JSPS KAKENHI Grant Number 19K14494.

\section{Notation and preliminaries}\label{sec:Not}
In this section we introduce notation, in particular the functors of induction and restriction, Tadi\'c's formula and  Jantzen decomposition.
\subsection{Notation}
Throughout this article, 
we fix a non-Archimedean locally compact field $F$  of characteristic zero with normalized absolute value $|\cdot|$. 
Let $G$ be the group of $F$-points of a connected reductive group defined over $F$, with the usual topology.
We will only consider smooth representations of $G$, 
that is, representations such that the stabilizer of every vector is an open subgroup of $G$ 
and we write $\Rep(G)$ for the category of smooth complex representations of $G$ of finite length. 
Denote by $\Irr(G)$ the set of equivalence classes of irreducible objects of $\Rep(G)$. 
Let $\Rr(G)$ be the Grothendieck group of $\Rep(G)$. 
The canonical map from the objects of $\Rep(G)$ to $\Rr (G)$ will be denoted by $\pi\mapsto[\pi]$.
\par

For $\pi, \pi' \in\Rep(G)$ we write $\pi \hookrightarrow \pi'$ (\resp $\pi \twoheadrightarrow \pi'$) 
if there exists an injective (\resp surjective) morphism from $\pi$ to $\pi'$.
\par

Fix a minimal $F$-parabolic subgroup $P_0$ of $G$. A parabolic subgroup $P$ of $G$ will be called \emph{standard} if it contains $P_0$. 
Henceforth, the letter $P$ will always denote a standard parabolic subgroup of $G$ 
with an implicit standard Levi decomposition $P=MU$. Let $\Sigma$ denote the set of roots of $G$ with respect to $P_0$ and let $\Delta$ be a basis of $\Sigma$. For $\Theta \subset \Delta$ let $P_\Theta$ denote the standard parabolic subgroup of $G$ corresponding to $\Theta$ and let $M_\Theta$ be a corresponding standard Levi subgroup. Let $W$ be the Weyl group of $G$.

Let $\tau$ be a representation of $M$, regarded as a representation of $P$ on which $U$ acts trivially.
We denote by $\Ind_P^G\tau$, the representation of $G$ parabolically induced from $\tau$. 
(We will always mean the normalized induction.)
We view $\Ind_P^G$ as a functor. 
Its left adjoint, the Jacquet functor with respect to $P$, will be denoted by $\Jac_P$.
\par

An irreducible representation $\pi$ of $G$ is called \emph{supercuspidal} 
if it is not a composition factor of any representation of the form
$\Ind^{G}_P(\tau)$ with $P$ a proper parabolic subgroup of $G$ and $\tau$ a representation of $M$. 
We  write $\Cusp(G)$ for the subset of  $\Irr(G)$ made of supercuspidal representations.
For any $\pi \in \Rep(G)$, we denote by $\pi^\vee$ the contragredient of $\pi$.
(The sets $\Irr(G)$ and $\Cusp(G)$ are invariant under $^\vee$.)
\par

Let $\Pi$ be a smooth representation of $G$ of finite length. 
The \emph{socle} of $\Pi$ is the largest semisimple subrepresentation of $\Pi$. 
It is denoted by $\soc(\Pi)$.
We say that $\Pi$ is \emph{socle irreducible (SI)} 
if $\soc(\Pi)$ is irreducible and occurs with multiplicity one in $[\Pi]$. 

\subsection{The Zelevinsky--Aubert duality}\label{sec:Au}
We consider the map
\begin{align*}
D_{G} \colon \Rr(G) & \longrightarrow  \Rr(G) \\
\pi &\mapsto  \sum_{P}(-1)^{\dim A_M} \left[\Ind_P^{G}(\Jac_{P}(\pi))\right],
\end{align*}
where $P = MN$ runs over all standard parabolic subgroups of $G$.
Then Aubert \cite{Au} showed that 
if $\pi$ is irreducible, 
then there exists a sign $\epsilon \in \{\pm1\}$ 
such that $\hat{\pi} = \epsilon \cdot D_{G}(\pi)$ is also an irreducible representation. 
We call the map 
\begin{align*}
\Irr(G) & \rightarrow  \Irr(G) \\
\pi& \mapsto \hat\pi
\end{align*}
the \emph{Zelevinsky--Aubert duality}. 
\par

It satisfies the following important properties: 
\begin{enumerate}
\item
For any $\pi \in  \Irr(G)$, the  dual of $\hat\pi$ is equal to $\pi$, 
that is, the map $\pi \mapsto \hat\pi$ is an involution \cite[Th{\'e}or{\`e}me 1.7 (3)]{Au}.
\item
If $\pi\in \Cusp(G)$, then $\hat\pi = \pi$ \cite[Th{\'e}or{\`e}me 1.7 (4)]{Au}.
\item Let $\Theta \subset \Delta$ and consider the standard parabolic subgroup $P=P_\Theta$ with Levi decomposition $P=MN$. Let $w_0$ be the longest element in the set $\{w\in W\;|\;w^{-1}(\Theta)>0\}$ and let $P'$ be the standard parabolic with Levi subgroup $M'=w^{-1}(M)$. Then  we have (\textit{cf.} \cite[Th{\'e}or{\`e}me 1.7.(2)]{Au}):
\begin{equation}\label{eq:Jacquet}
\Jac_P\circ D_G= {\rm Ad}(w_0)\circ D_{M'} \circ \Jac_{P'}.
\end{equation}
\end{enumerate}

\subsection{Representations of general linear groups}\label{rep0}
Set $\Irr^\GL \coloneqq \cup_{n\ge0}\Irr (\GL_n(F))$ 
and let $\Cusp^\GL \subset \Irr^\GL$ be the subset of supercuspidal representations of $\GL_n(F)$ for every $n>0$. 
We denote $\Rr^\GL \coloneqq \oplus_{n \geq 0} \Rr(\GL_n(F))$.
\par

Let $d_1, \dots, d_r$ be some positive integers. 
Let  $\tau_i \in \Rep(\GL_{d_i}(F))$ for $1 \leq i \leq r$. 
It is customary to denote the normalized parabolically induced representation by 
\[
\tau_1 \times \dots \times \tau_r \coloneqq \Ind_{P}^{\GL_{k_1+\dots+k_r}(F)}(\tau_1 \boxtimes \dots \boxtimes \tau_r). 
\]
This product  induces a $\Z$-graded ring structure on $\Rr^\GL$. 
We denote the multiplication by $m$. 
If $\tau_1 = \dots = \tau_r = \tau$, we will write $\tau^r = \tau \times \dots \times \tau$ ($r$ times).
\par

The Jacquet functor for $\GL_{m}(F)$ along the maximal parabolic subgroup $P_{(d,m-d)}$ 
with Levi isomorphic to $\GL_{d}(F) \times \GL_{m-d}(F)$
is denoted by $\Jac_{(d,m-d)}$. It induces a co-multiplication, that is a ring homomorphism
\begin{align*}
m^\ast \colon \Rr^\GL &\longrightarrow \Rr^\GL\otimes \Rr^\GL \notag\\
\tau &\mapsto \sum_{n \geq 0} \left(\sum_{n_1+n_2=n}\left[\Jac_{(n_1,n_2)}(\tau)\right]\right)
\end{align*}
We finally set
\[
M^\ast \colon \Rr^\GL \longrightarrow \Rr^\GL\otimes \Rr^\GL 
\]
to be the composition $M^\ast = (m\otimes 1) \circ (\cdot^\vee\otimes m^\ast) \circ s \circ m^\ast$, 
where $s \colon \Rr^\GL\otimes \Rr^\GL \to \Rr^\GL\otimes \Rr^\GL $ denotes 
the transposition $s(\sum_i \tau_i\otimes \tau'_i)=\sum_i \tau'_i\otimes \tau_i$.
\par

If $\tau \in \Irr^\GL$, there exist $\rho_1, \dots, \rho_r \in \Cusp^\GL$ 
such that $\tau$ is a subrepresentation of $\rho_1 \times \dots\times \rho_r $. 
The set $\scusp(\pi) \coloneqq \{\rho_1, \dots, \rho_r \}$ is uniquely determined by $\pi$ 
and is called the \emph{supercuspidal support} of $\tau$.
\par

For  $\pi \in \Rep(\GL_n(F))$ and a character $\chi$ of $F^\times$, 
we denote by $\pi\cdot\chi$ the representation obtained from $\pi$ by twisting by the character $\chi \circ \det$. 
If $\rho \in \Cusp^\GL$, we denote by $\Z_\rho = \{\rho|\cdot|^a \;|\; a\in \Z\}$ \emph{the line of $\rho$}.
\par

A \emph{segment} $[x,y]_\rho$ is a sequence of supercuspidal representations of the form
\[
\rho|\cdot|^{x} ,\rho|\cdot|^{x-1} ,\dots ,\rho|\cdot|^{y},
\]
where $\rho \in \Cusp^\GL$ and $x,y \in \R$ with $x-y \in \Z$ and $x \geq y$.
\par

One can associate to a segment $[x,y]_\rho$ two irreducible representations of $\GL_{d(x-y+1)}(F)$. 
We denote by $\Delta_{\rho}[x,y]$ the \emph{Steinberg representation} of $\GL_{d(x-y+1)}(F)$, i.e., 
the unique irreducible subrepresentation of 
\[
\rho|\cdot|^{x} \times \rho|\cdot|^{x-1} \times \dots \times \rho|\cdot|^{y}, 
\]
and we also write $Z_\rho[y,x] $ for its unique irreducible quotient. 
For example, when $\rho = \1_{\GL_1(F)}$, we have $Z_\rho[-(n-1)/2,(n-1)/2] = \1_{\GL_n(F)}$. 
\par

The Steinberg representation  $\Delta_{\rho}[x,y]$ is an essentially discrete series 
and all essentially discrete series are of this form \cite[Theorem 9.3]{Z}. 
By convention, we set $\Delta_\rho[x,x+1]= Z_\rho[x+1,x]$ to be the trivial representation of the trivial group $\GL_0(F)$.
\par

If the segments $[x_1,y_1]_{\rho_1}, \dots, [x_r,y_r]_{\rho_r}$ satisfy that 
$x_i \geq y_i$ and $x_1+y_1 \leq \dots \leq x_r+y_r$, 
then the socle 
\[
L(\Delta_{\rho_1}[x_1,y_1], \dots, \Delta_{\rho_r}[x_r,y_r]) 
\coloneqq \soc(\Delta_{\rho_1}[x_1,y_1] \times \dots \times \Delta_{\rho_r}[x_r,y_r])
\]
is irreducible.
When $\rho_1 = \dots = \rho_r$, 
$x_1 < \dots < x_r$, $y_1 < \dots < y_r$ and $x_1 \equiv \dots \equiv x_r \bmod \Z$, 
we call it a \emph{ladder representation}. 
As a special case, when $x_i = x_1+i-1$ and $y_i = y_1+i-1$ for $1 \leq i \leq r$, 
the ladder representation $L(\Delta_\rho[x_1,y_1], \dots, \Delta_\rho[x_r,y_r])$ 
is also called a \emph{Speh representation}.
\par

The Jacquet modules of $\Delta_\rho[x,y]$ and $Z_\rho[y,x]$ are given by
\begin{align*}
\Jac_{(d,d(x-y))}(\Delta_{\rho}[x,y]) &= \rho|\cdot|^x \boxtimes \Delta_\rho[x-1,y], \\
\Jac_{(d,d(x-y))}(Z_{\rho}[y,x]) &= \rho|\cdot|^y \boxtimes Z_\rho[y+1,x], 
\end{align*}
respectively (see \cite[Propositions 3.4, 9.5]{Z}). 
For Jacquet modules of ladder representations, see \cite[Theorem 2.1]{KL}. 
\par

\subsection{Representations of classical groups}\label{rep}
In this paper, we let $G_n$ be either the split special orthogonal group $\SO_{2n+1}(F)$ 
or the symplectic group $\Sp_{2n}(F)$ of rank $n$. 
Set $\Irr^{G} \coloneqq \cup_{n \geq 0} \Irr (G_n)$ and $\Rr^G \coloneqq \oplus_{n \geq 0} \Rr(G_{n})$, 
where the union and the direct sum are taken over groups of the same type. 
Let $\Cusp^G \subset \Irr^G$ be the subset of supercuspidal representations of $G_n$ for every $n \geq 0$ of the same type. 
\par

Fix a rational Borel subgroup of $G_n$. 
Let $P$ be the standard parabolic subgroup of $G_n$ with Levi subgroup isomorphic to
$\GL_{d_1}(F) \times \dots \times \GL_{d_r}(F) \times G_{n_0}$. 
Let $\pi\in \Rep(G_{n_0})$ and let $\tau_i \in \Rep(\GL_{d_i}(F))$ for $1 \leq i \leq r$. 
We denote the normalized parabolically induced representation by
\[
\tau_1 \times \dots \times \tau_r \rtimes \pi 
\coloneqq 
\Ind_{P}^{G_n}(\tau_1 \boxtimes \dots \boxtimes \tau_r \boxtimes \pi).
\]
As in the case of general linear groups, 
the Jacquet functors give rise, at the level of Grothendieck groups, to a map
\begin{align*}
\mu^\ast \colon \Rr^G &\longrightarrow \Rr^\GL\otimes \Rr^G, \\
\Rr(G_n) \ni \pi &\longmapsto \sum_{k=0}^{n} \left[\Jac_{P_k}(\pi)\right], 
\end{align*}
where $P_k$ is the standard parabolic subgroup of $G_n$ 
with Levi subgroup isomorphic to $\GL_{k}(F) \times G_{n-k}$. 
The Geometric Lemma at the level of Grothendieck groups is commonly known in this case 
as \emph{Tadi{\'c}'s formula}.
\begin{prop}[Tadi{\'c}'s formula \cite{T}]\label{eq:Tadic}
For $\tau \in \Rr^{\GL}$ and $\pi \in \Rr^G$, we have
\[
\mu^\ast(\tau \rtimes \pi)=M^\ast(\tau)\rtimes\mu^\ast(\pi).
\]
\end{prop}
\par

We will also use the \emph{MVW-functor} \cite{MVW}. It is a covariant functor 
\begin{align*}
\MVW \colon \Rep(G_n) &\longrightarrow \Rep(G_n), \\ 
\Pi &\mapsto \Pi^\MVW
\end{align*}
satisfying the following properties:
\begin{itemize}
\item
if $\pi \in \Irr (G_n)$, then  $\pi^\MVW$ is isomorphic to $\pi^\vee$; 
\item
we have $(\tau \rtimes \pi)^\MVW \cong \tau \rtimes \pi^\MVW$ for any  $\pi\in \Rep(G_{n_0})$ 
and any $\tau \in \Rep(\GL_{d}(F))$ with $n=n_0+d$.
\end{itemize}
\par
The Zelevinsky--Aubert duality extends by linearity to a map $D^G: \Rr^G \rightarrow \Rr^G$. 
With this notation, the compatibility of the duality with Jacquet functors in equation \eqref{eq:Jacquet} stands:
\begin{equation}\label{eq:Jacquet2}
\mu^\ast \circ D^G =  d^G  \circ  \mu^\ast, 
\end{equation}
where
\begin{align*}
d^G \colon  \Rr^\GL \otimes \Rr^G &\longrightarrow \Rr^\GL\otimes \Rr^G \\
\sum_i \tau_i \otimes \pi_i &\longmapsto \sum_i \hat{\tau}_i^\vee \otimes \hat{\pi}_i.
\end{align*}
\par

Let $[x_1,y_1]_{\rho_1}, \dots, [x_r,y_r]_{\rho_r}$ be some segments 
with $\rho_i \in \Cusp(\GL_{d_i}(F))$ being unitary for $1 \leq i \leq r$, 
and let $\pi_\temp$ be an irreducible tempered representation of $G_{n_0}$.
A parabolically induced representation of the form
\[
\Delta_{\rho_1}[x_1,y_1] \times \dots \times \Delta_{\rho_r}[x_r,y_r] \rtimes \pi_\temp 
\]
is called a \emph{standard module}
if $x_1+y_1 \leq \dots \leq x_r+y_r < 0$. 
\par

The Langlands classification says that any standard module is SI, 
and that any irreducible representation $\pi$ of $G_n$ is the unique irreducible subrepresentation of 
a standard module $\Delta_{\rho_1}[x_1,y_1] \times \dots \times \Delta_{\rho_r}[x_r,y_r] \rtimes \pi_\temp$ 
with $n=n_0 + \sum_{i=1}^r d_i(x_i-y_i+1)$, which is unique up to isomorphism. 
For more details, see \cite{K}. 
In this case, we write $\pi = L(\Delta_{\rho_1}[x_1,y_1], \dots, \Delta_{\rho_r}[x_r,y_r]; \pi_\temp)$, 
and refer $(\Delta_{\rho_1}[x_1,y_1], \dots, \Delta_{\rho_r}[x_r,y_r]; \pi_\temp)$ 
as the \emph{Langlands data} of $\pi$.
\par

\subsection{The Jantzen decomposition}\label{sub:Jantzen}
If $\pi\in \Irr(G_n)$, 
there exist $\rho_1, \dots, \rho_r \in \Cusp^\GL$ and $\sigma \in \Cusp^G$ 
such that $\pi$ is a subrepresentation of $\rho_1 \times \dots\times \rho_r \rtimes\sigma$. 
The set 
\[
\scusp(\pi) \coloneqq \{\rho_1, \dots, \rho_r, \rho_1^\vee, \dots, \rho_r^\vee ,\sigma\} 
\]
is uniquely determined by $\pi$ 
and is called the \emph{supercuspidal support} of $\pi$. 
For $\sigma \in \Cusp^G$, we put $\Irr_\sigma \coloneqq \{\pi\in\Irr^G \;|\; \sigma \in \scusp(\pi)\}$.
\par

In this paragraph, we fix a supercuspidal representation $\sigma \in \Cusp^G$.
\begin{defi}\label{def:good}
Let $\rho \in \Cusp^\GL$.
\begin{itemize}
\item
We say $\rho$ is \emph{good} 
if $\Z_\rho = \Z_{\rho^\vee}$ and $\rho' \rtimes \sigma$ is reducible for some $\rho' \in \Z_\rho$. 
\item
We say $\rho$ is \emph{bad} 
if $\Z_\rho = \Z_{\rho^\vee}$ and $\rho' \rtimes \sigma$ is irreducible for all $\rho' \in \Z_\rho$.
\item
We say $\rho$ is \emph{ugly} 
if $\Z_\rho \neq \Z_{\rho^\vee}$.
\end{itemize}
\end{defi}
Every supercuspidal representation is either good, bad or ugly.
\begin{rem}
It is known that
\begin{itemize}
\item
the notions of good and bad are independent of $\sigma$;

\item
if $\rho'|\cdot|^z $ is  good or bad  with $\rho'$ unitary and $z\in\R$, then $\rho'$ is self-dual and $z \in (1/2)\Z$; 

\item
if $\rho'|\cdot|^{z_1}, \rho'|\cdot|^{z_2} $ are both good or both bad, then $z_1-z_2 \in \Z$. 
\end{itemize}
See Remark \ref{rem:good} below. 
\end{rem}

\begin{defi}
\begin{enumerate}
\item
We say two good (\resp bad) supercuspidal representations $\rho, \rho'$ are \emph{line equivalent} 
if $\Z_\rho=\Z_{\rho'}$. 
We denote by $\Cusp^{\good}$ (\resp $\Cusp^{\bad}$) 
a set of representatives of good (\resp bad) representations under this equivalence relation.

\item 
Similarly we say two ugly representations $\rho, \rho'$ are \emph{line equivalent} 
if $\Z_\rho\cup\Z_{\rho^\vee} = \Z_{\rho'}\cup\Z_{\rho^{'\vee}}$. 
We denote by $\Cusp^{\ugly}$ a set of representatives of ugly representations under this equivalence relation.
\end{enumerate}
\end{defi}

\begin{defi}
Let $\pi \in \Irr_\sigma$.
\begin{enumerate}
\item
If 
\[
\scusp(\pi) \subset \left( \bigcup_{\rho \in \Cusp^{\good}} \Z_\rho \right) \cup \{\sigma\}, 
\]
we say that $\pi$ is \emph{of good parity}. 
We write  $\Irr_\sigma^\good$ for the set of such representations.

\item 
If $\scusp(\pi) \subset \Z_\rho \cup \{\sigma\}$ for some bad representation $\rho$, 
we say that $\pi$ is \emph{of bad parity} (or of $\rho$-bad parity if we want to specify $\rho$). 
We write  $\Irr_\sigma^{\rho-\bad}$ for the set of such representations.

\item 
If $\scusp(\pi) \subset (\Z_\rho \cup \Z_{\rho^\vee}) \cup \{\sigma\}$ for some ugly representation $\rho$, 
we say that $\pi$ is \emph{ugly} (or $\rho$-ugly if we want to specify $\rho$). 
We write  $\Irr_\sigma^{\rho-\ugly}$ for the set of such representations.
\end{enumerate}
\end{defi}

Ugly representations are easy to deal with due to the following proposition 
which reduces every problem to a similar problem for general linear groups.
\begin{prop}\label{prop:ugly}
Let $\pi \in \Irr_\sigma^{\rho-\ugly}$. 
Then there exists an irreducible representation $\tau$ of $\GL_m(F)$ with $\scusp(\tau) \subset \Z_\rho$
such that $\pi = \tau \rtimes \sigma$ (irreducible induction). 
\end{prop}
\begin{proof}
We can write 
$\pi \hookrightarrow 
\rho|\cdot|^{x_1} \times \dots \times \rho|\cdot|^{x_r} 
\times
\rho^\vee|\cdot|^{-y_1} \times \dots \times \rho^\vee|\cdot|^{-y_s} 
\rtimes \sigma$
for some $x_i, y_j \in \Z$.
There exist irreducible subquotients $\tau_1$ of $\rho|\cdot|^{x_1} \times \dots \times \rho|\cdot|^{x_r}$
and $\tau_2$ of $\rho^\vee|\cdot|^{-y_1} \times \dots \times \rho^\vee|\cdot|^{-y_s}$ 
such that this inclusion factors through $\pi \hookrightarrow \tau_1 \times \tau_2 \rtimes \sigma$.
As $\rho$ is ugly, we can apply \cite[Lemma 6.2]{LT} to $\tau_2 \rtimes \sigma$, 
and we see that $\tau_2 \rtimes \sigma$ is irreducible. 
Hence $\pi \hookrightarrow \tau_1 \times \tau_2^\vee \rtimes \sigma$. 
Take an irreducible subquotient $\tau$ of $\tau_1 \times \tau_2^\vee$ 
such that $\pi \hookrightarrow \tau \rtimes \sigma$. 
Then by \cite[Lemma 6.2]{LT} again, we conclude that $\tau \rtimes \sigma$ is irreducible.
\end{proof}

\begin{rem}\label{rem:ugly}
More precisely, by the Langlands classification, 
one can take $\tau_1$, $\tau_2$ in the proof of this proposition so that 
\[
\tau_1 = L(\Delta_\rho[x'_1, y'_1], \dots, \Delta_\rho[x'_{r'},y'_{r'}]), 
\quad
\tau_2 = L(\Delta_{\rho^\vee}[x_1'',y_1''], \dots, \Delta_{\rho^\vee}[x''_{r''},y''_{r''}]) 
\]
with $x'_1+y'_1 \leq \dots \leq x'_{r'}+y'_{r'} \leq 0$ and $x''_1+y''_1 \leq \dots \leq x''_{r''}+y''_{r''} \leq 0$.
Then since $\tau_2^\vee = L(\Delta_{\rho}[-y''_{r''},-x''_{r''}], \dots, \Delta_{\rho}[-y''_{1},-x''_{1}])$
and since 
$\pi = \soc(\tau_1 \times \tau_2^\vee \rtimes \sigma) \hookrightarrow \soc(\tau_1 \times \tau_2^\vee) \rtimes \sigma$, 
one can take $\tau$ as
\[
\tau \coloneqq \soc(\tau_1 \times \tau_2^\vee) 
= L(\Delta_\rho[x'_1, y'_1], \dots, \Delta_\rho[x'_{r'},y'_{r'}], \Delta_{\rho}[-y''_{r''},-x''_{r''}], \dots, \Delta_{\rho}[-y''_{1},-x''_{1}]).
\]
\end{rem}
\par

Let $\pi\in \Irr_\sigma$. 
Then Jantzen \cite{J-dec} defines representations 
$\pi^\good \in \Irr_\sigma^\good$, $\pi^{\rho-\bad} \in \Irr_\sigma^{\rho-\bad}$ 
and $\pi^{\rho-\ugly} \in \Irr_\sigma^{\rho-\ugly}$ as follows. 
\begin{itemize}
\item 
$\pi^\good$ is the unique representation in $ \Irr_\sigma^\good$ 
such that $\pi \hookrightarrow \tau \times \pi^\good$ with no good representations in $\scusp(\tau)$.
\item 
If $\rho$ is a bad supercuspidal representation, 
then $\pi^{\rho-\bad}$ is the unique representation in $\Irr_\sigma^{\rho-\bad}$ 
such that $\pi \hookrightarrow \tau \times\pi^{\rho-\bad }$ with $\scusp(\tau) \cap \Z_\rho=\emptyset$.
\item 
If $\rho$ is an ugly supercuspidal representation, 
then $\pi^{\rho-\ugly}$ is the unique representation in $\Irr_\sigma^{\rho-\ugly}$ 
such that $\pi \hookrightarrow \tau \times \pi^{\rho-\ugly}$ 
with $\scusp(\tau) \cap (\Z_\rho\cup\Z_{\rho^\vee}) = \emptyset$.
\end{itemize}

The following theorem is a special case of Jantzen's decomposition.
\begin{thm}[{\cite[Theorem 9.3]{J-dec}}] The map
\begin{align*}
\Psi \colon \Irr_\sigma &\longrightarrow 
\Irr_\sigma^\good 
\sqcup \left(\bigsqcup_{\rho\in \Cusp^{\bad}} \Irr_\sigma^{\rho-\bad}\right)  
\sqcup \left(\bigsqcup_{\rho\in \Cusp^{\ugly}} \Irr_\sigma^{\rho-\ugly}\right), \\
\pi &\longmapsto \left( \pi^\good, \{\pi^{\rho-\bad}\}_\rho, \{\pi^{\rho-\ugly}\}_\rho \right) 
\end{align*}
is bijective. Moreover, it commutes with the Zelevinsky--Aubert duality in the sense:
$$\Psi(\hat\pi)=\left( \widehat{\pi^\good}, \{\widehat{\pi^{\rho-\bad}}\}_\rho, \{ \widehat{\pi^{\rho-\ugly}}\}_\rho \right). $$
\end{thm}
In practice, this theorem enables us to reduce the problem of making the Zelevinsky--Aubert duality explicit 
to the case where the representation is either ugly or of good or bad parity. 

\section{The theory of $\rho$-derivatives}\label{sec:deriv}
Let $d>0$ be an integer. 
In this section, we fix $\rho \in \Cusp(\GL_d(F))$.
We recall $\rho$-derivatives as in \cite{LT}
and introduce the notions of $\Delta_\rho[0,-1]$-derivative and $Z_\rho[0,1]$-derivative.
One should not confuse these notions with Bernstein--Zelevinsky's notion of derivatives.

\subsection{Definitions}
We treat first the case of general linear groups. 
For $\tau  \in\Rep(\GL_n(F))$,
define semisimple representations $L_\rho^{(k)}(\tau)$ and $R_\rho^{(k)}(\tau)$ of $\GL_{n-dk}(F)$ so that
\begin{align*}
\left[\Jac_{(dk,n-dk)}(\tau)\right] &= \rho^k \boxtimes L_{\rho}^{(k)}(\tau) + \sum_i \tau_i \boxtimes \sigma_i, \\
\left[\Jac_{(n-dk,dk)}(\tau)\right] &= R_{\rho}^{(k)}(\tau) \boxtimes \rho^k + \sum_i \sigma'_i \boxtimes \tau'_i,
\end{align*}
where $\tau_i$ and $\tau'_i$ are irreducible representations of $\GL_{dk}(F)$ 
which are not isomorphic to $\rho^k$.
We call $L_\rho^{(k)}(\tau)$ (\resp $R_\rho^{(k)}(\tau)$) the \emph{$k$-th left $\rho$-derivative} 
(\resp the \emph{$k$-th right $\rho$-derivative}) of $\tau$. 

\begin{defi}\label{def-der1}
\begin{enumerate}
\item
If $L_{\rho}^{(k)}(\tau) \not= 0$ but $L_{\rho}^{(k+1)}(\tau) =0$, 
we say that $L_{\rho}^{(k)}(\tau)$ is the \emph{highest left $\rho$-derivative}. 
We also define the \emph{highest right $\rho$-derivative} similarly. 

\item
When $L_{\rho}^{(1)}(\tau) =0$ (\resp $R_{\rho}^{(1)}(\tau) =0$), 
we say that $\tau$ is \emph{left $\rho$-reduced} (\resp \emph{right $\rho$-reduced}).
\end{enumerate}
\end{defi}

Similarly we treat now the case of $G_n$.
Again let $k \geq 0$ and let $P_{dk}$ be now the standard parabolic subgroup of $G_n$ 
with Levi subgroup of the form $\GL_{dk}(F) \times G_{n-dk}$. 
For $\Pi  \in\Rep(G_n)$,
define a semisimple representation $D_\rho^{(k)}(\Pi)$ of $G_{n-dk}$ so that
\[
\left[\Jac_{P_{dk}}(\Pi)\right] = \rho^k \boxtimes D_{\rho}^{(k)}(\Pi) + \sum_i \tau_i \boxtimes \Pi_i, 
\]
where $\tau_i$ is an irreducible representation of $\GL_{dk}(F)$ which is not isomorphic to $\rho^k$.
We call $D_\rho^{(k)}(\Pi)$ the \emph{$k$-th $\rho$-derivative} of $\Pi$. 

\begin{defi}\label{def-der}
\begin{enumerate}
\item
If $D_{\rho}^{(k)}(\Pi) \not= 0$ but $D_{\rho}^{(k+1)}(\Pi) =0$, 
we say that $D_{\rho}^{(k)}(\Pi)$ is the \emph{highest $\rho$-derivative}. 

\item
When $D_{\rho}^{(1)}(\Pi) =0$, 
we say that $\Pi$ is \emph{$\rho$-reduced}.
\end{enumerate}
\end{defi}

\subsection{The non-self-dual case}
If $\pi$ is irreducible and if $\rho$ is not self-dual, 
then the highest $\rho$-derivative $D_\rho^{(k)}(\pi)$ is irreducible and 
$\pi$ is isomorphic to the unique irreducible subrepresentation of $\rho^k \rtimes D_{\rho}^{(k)}(\pi)$
(see \cite[Lemma 3.1.3]{J-temp} and \cite[Proposition 2.7]{At2}). 
Using these properties, we can show the following. 

\begin{prop}\label{SI}
Let $\pi$ be an irreducible representation of $G_n$, 
and $r$ be a non-negative integer. 
If  $\rho$ is not self-dual, 
then $\rho^r \rtimes \pi$ is SI. 
\end{prop}
\begin{proof}
Consider the highest $\rho$-derivative $D_{\rho}^{(k)}(\pi)$. 
If $\pi' \hookrightarrow \rho^r \rtimes \pi$, 
then $\pi' \hookrightarrow \rho^{k+r} \rtimes D_{\rho}^{(k)}(\pi)$. 
In particular, $D_{\rho}^{(k+r)}(\pi') = D_{\rho}^{(k)}(\pi)$.
However, since 
\[
D_{\rho}^{(k+r)}\left(
\rho^{k+r} \rtimes D_{\rho}^{(k)}(\pi)
\right)
= D_{\rho}^{(k)}(\pi)
\]
by Tadi{\'c}'s formula (Proposition \ref{eq:Tadic}), 
we see that $\pi'$ is determined uniquely. 
Hence $\soc(\rho^r \rtimes \pi)$ is irreducible and satisfies 
\[
D_{\rho}^{(k+r)}\left(
\soc(\rho^r \rtimes \pi)
\right)
=
D_{\rho}^{(k+r)}\left(
\rho^r \rtimes \pi
\right)
=
D_{\rho}^{(k)}(\pi). 
\]
These equations imply that 
$\soc(\rho^r \rtimes \pi)$ appears with multiplicity one in $\left[\rho^r \rtimes \pi\right]$. 
\end{proof}

We set 
\[
S_\rho^{(r)}(\pi) \coloneqq \soc(\rho^r \rtimes \pi)
\]
for any $\pi\in \Irr(G_n)$. 
Note that $S_\rho^{(r)} = S_\rho^{(1)} \circ \dots \circ S_\rho^{(1)}$ ($r$ times compositions). 
\par

\subsection{The self-dual case}
Recall in \cite[Proposition 2.7]{At2} that 
the highest $\rho$-derivative $D_\rho^{(k)}(\pi)$ of an irreducible representation is isotypic, 
i.e., $D_\rho^{(k)}(\pi) = m \cdot \pi_0$ with some irreducible representation $\pi_0$ and a certain multiplicity $m > 0$. 
In this case, we have $\pi \hookrightarrow \rho^k \rtimes \pi_0$, 
but $\soc(\rho^k \rtimes \pi_0)$ can be reducible. 
\par

We give a criterion where $\rho^r \rtimes \pi$ is SI.
\begin{prop}
Suppose that $\rho$ is self-dual.
Let $\pi\in \Irr(G_n)$, 
and $r$ be a positive integer. 
The following are equivalent. 
\begin{enumerate}
\item[(a)]
$\rho^r \rtimes \pi$ is SI; 
\item[(b)]
$\rho^r \rtimes \pi$ is irreducible; 
\item[(c)]
$\rho^r \rtimes \pi$ has an irreducible subquotient $\pi'$ such that 
$D_{\rho}^{(k+r)}(\pi') = 2^r \cdot D_{\rho}^{(k)}(\pi)$, 
where $D_{\rho}^{(k)}(\pi)$ is the highest $\rho$-derivative of $\pi$. 
\end{enumerate}
\end{prop}
\begin{proof}
We use here the MVW-functor, see Paragraph \ref{rep}. 
As we assume that $\rho$ is self-dual, 
if an irreducible representation $\pi'$ satisfies that $\pi' \hookrightarrow \rho^r \rtimes \pi$, 
then by taking the MVW-functor and the contragredient functor, 
we have $\rho^r \rtimes \pi \twoheadrightarrow \pi'$. 
\par

Now we assume that $\soc(\rho^r \rtimes \pi)$ is irreducible but $\rho^r \rtimes \pi$ is reducible. 
The above remark implies that the quotient $(\rho^r \rtimes \pi) / \soc(\rho^r \rtimes \pi)$ 
has an irreducible quotient isomorphic to $\soc(\rho^r \rtimes \pi)$.
It means that $\soc(\rho^r \rtimes \pi)$ appears with multiplicity greater than one in $\left[\rho^r \rtimes \pi\right]$. 
Hence (a) implies (b).
As the opposite implication is obvious,  (a) and (b) are equivalent. 
\par

Note that $D_{\rho}^{(k+r)}(\rho^r \rtimes \pi) = 2^r \cdot D_{\rho}^{(k)}(\pi)$.
In particular, (b) implies (c).
On the other hand, let $\pi'$ be an irreducible subquotient of $\rho^r \rtimes \pi$ such that 
$D_{\rho}^{(k+r)}(\pi') = 2^r \cdot D_{\rho}^{(k)}(\pi)$.
Then $\pi'$ must be a subrepresentation of $\rho^r \rtimes \pi$, 
and $(\rho^r \rtimes \pi)/\pi'$ has no irreducible quotient. 
Hence $\pi' = \rho^r \rtimes \pi$ so that $\rho^r \rtimes \pi$ is irreducible. 
\end{proof}

\subsection{$\Delta_\rho[0,-1]$-derivatives and $Z_\rho[0,1]$-derivatives}\label{other-derivatives}
In the case when $\rho$ is self-dual, $\rho$-derivatives are difficult. 
Therefore, we define some other derivatives in this paragraph. 
This will be a key ingredient for the making the Zelevinsky--Aubert duality explicit.
In this paragraph we assume that $\rho \in \Cusp(\GL_d(F))$ is self-dual.
\par

Let $\Pi \in \Rep(G_n)$.  
Define the \emph{$\Delta_\rho[0,-1]$-derivative} $D_{\Delta_\rho[0,-1]}^{(k)}(\Pi)$ 
and the \emph{$Z_\rho[0,1]$-derivative} $D_{Z_\rho[0,1]}^{(k)}(\Pi)$ by 
the semisimple representations of $G_{n-2dk}$ satisfying 
\[
\left[\Jac_{P_{2dk}}(\pi)\right] = 
\Delta_\rho[0,-1]^k \boxtimes D_{\Delta_\rho[0,-1]}^{(k)}(\pi)
+Z_\rho[0,1]^k \boxtimes D_{Z_\rho[0,1]}^{(k)}(\pi)
+ \sum_i \tau_i \boxtimes \pi_i, 
\]
where $\tau_i \in \Irr(\GL_{2dk}(F))$ such that 
$\tau_i \not\cong \Delta_\rho[0,-1]^k, Z_\rho[0,1]^k$. 
\par

Typically, when the supercuspidal representation $\rho$ will be clear from the context, 
for short, we say the \emph{$[0,-1]$-derivative} instead of the $\Delta_\rho[0,-1]$-derivative, 
and the \emph{$[0,1]$-derivative} instead of the $Z_\rho[0,1]$-derivative. 
We also write $D_{[0,-1]}^{(k)}(\Pi) \coloneqq D_{\Delta_\rho[0,-1]}^{(k)}(\Pi)$ 
and $D_{[0,1]}^{(k)}(\Pi) \coloneqq D_{Z_\rho[0,1]}^{(k)}(\Pi)$. 
Similar to Definition \ref{def-der}, we define 
the notions of the \emph{highest $[0,-1]$-derivatives} (\resp \emph{highest $[0,1]$-derivatives})
and the fact of being \emph{$\Delta_\rho[0,-1]$-reduced} (\resp \emph{$Z_\rho[0,1]$-reduced}).

\begin{lem}\label{[01]-red}
Fix $\rho\in\Cusp(\GL_d(F))$ and $\epsilon \in \{\pm 1\}$. 
Let $\pi\in \Irr(G_n)$. 
Suppose that $\pi$ is $\rho|\cdot|^{\epsilon}$-reduced. 
Let $D_{\rho}^{(k_0)}(\pi) = m \cdot \pi_0$ be the highest $\rho$-derivative of $\pi$ (with multiplicity $m > 0$) and let $\pi_1 = D_{\rho|\cdot|^{\epsilon}}^{(k_1)}(\pi_0)$
be the highest $\rho|\cdot|^{\epsilon}$-derivative of $\pi_0$. 
Then we have the following. 
\begin{enumerate}
\item
$k_0 \geq k_1$. 
\item
$D_{[0,\epsilon]}^{(k_1)}(\pi)$ is the highest $[0,\epsilon ]$-derivative. 
\item
$D_{[0,\epsilon]}^{(k_1)}(\pi)$ is $\rho|\cdot|^{\epsilon }$-reduced. 
\end{enumerate}
\end{lem}
\begin{proof}
Note that $\pi \hookrightarrow \rho^{k_0} \times (\rho|\cdot|^{\epsilon})^{k_1} \rtimes \pi_1$.
If $k_1 > k_0$, then no irreducible subquotient of $\rho^{k_0} \times (\rho|\cdot|^{\epsilon })^{k_1}$ 
is  left $\rho|\cdot|^{\epsilon}$-reduced. 
Since $\pi$ is $\rho|\cdot|^{\epsilon}$-reduced, 
we must have $k_0 \geq k_1$ and 
\[
\pi \hookrightarrow 
\left\{
\begin{aligned}
Z_\rho[0,1]^{k_1} \times \rho^{k_0-k_1} \rtimes &\pi_1 \iif \epsilon = 1, \\
\Delta_\rho[0,-1]^{k_1} \times \rho^{k_0-k_1} \rtimes &\pi_1 \iif \epsilon = -1.
\end{aligned}
\right. 
\]
\par

Now we claim that $\pi_1$ is $\rho$-reduced. 
This is trivial when $k_1 = 0$. 
If $k_1 > 0$ and  $\pi_1$ is not $\rho$-reduced, since $\pi_0$ is $\rho$-reduced, 
we can find a representation $\pi_1' \not= 0$ such that
\[
\pi_0 \hookrightarrow 
\left\{
\begin{aligned}
\Delta_\rho[1,0] \rtimes &\pi_1' \iif \epsilon = 1, \\
Z_\rho[-1,0] \rtimes &\pi_1' \iif \epsilon = -1.
\end{aligned}
\right. 
\]
Since $\pi \hookrightarrow \rho^{k_0} \rtimes \pi_0$, 
it implies that $D_{\rho|\cdot|^{\epsilon}}^{(1)}(\pi) \not= 0$. 
This is a contradiction so that we obtain the claim.
\par

Since $\pi_1$ is $\rho$-reduced and $\rho|\cdot|^{\epsilon }$-reduced, 
we see that $D_{[0,\epsilon ]}^{(1)}(\rho^{k_0-k_1} \rtimes \pi_1) = 0$ by Tadi{\'c}'s formula (Proposition \ref{eq:Tadic}).
Hence $D_{[0,\epsilon ]}^{(k_1)}(\pi)$ is the highest $[0,\epsilon ]$-derivative. 
Since it is a subrepresentation of $\left[\rho^{k_0-k_1} \rtimes \pi_1\right]$,
we see that $D_{[0, \epsilon ]}^{(k_1)}(\pi)$ is $\rho|\cdot|^{\epsilon }$-reduced.
\end{proof}

In the next proposition, we will use the following simple lemma on representations of general linear groups.
\begin{lem}\label{surj}
Let $k>0$ and let $\tau\in \Rep(\GL_{2dk}(F))$. 
Suppose that
\begin{itemize}
\item
$\tau$ is left $\rho|\cdot|^{-1}$-reduced (\resp left $\rho|\cdot|^1$-reduced); 
\item
$[\tau]$ contains $\Delta_\rho[0,-1]^k$ (\resp $Z_\rho[0,1]^k$).  
\end{itemize}
Then there is a surjection $\tau \twoheadrightarrow \Delta_\rho[0,-1]^k$ (\resp $\tau \twoheadrightarrow Z_\rho[0,1]^k$).
\end{lem}
\begin{proof}
We may assume that all irreducible constituents of $\tau$ have the same supercuspidal support. 
They are all left $\rho|\cdot|^{-1}$-reduced (\resp left $\rho|\cdot|^1$-reduced) as is $\tau$.
By \cite[Example 11.3]{Z}, 
the irreducible representations of $\GL_{2dk}(F)$ which have the same supercuspidal support as $\Delta_\rho[0,-1]^k$
(\resp $Z_\rho[0,1]^k$) are of the form $\Delta_\rho[0,-1]^a \times Z_\rho[-1,0]^b$ 
(\resp $\Delta_\rho[1,0]^a \times Z_\rho[0,1]^b$) for some $a,b \geq 0$ with $a+b = k$. 
Among  them, $\Delta_\rho[0,-1]^k$ (\resp $Z_\rho[0,1]^k$) is characterized as 
the only left $\rho|\cdot|^{-1}$-reduced (\resp left $\rho|\cdot|^1$-reduced) representation. 
Therefore, we have $\tau \twoheadrightarrow \Delta_\rho[0,-1]^k$ 
(\resp $\tau \twoheadrightarrow Z_\rho[0,1]^k$). 
\end{proof}

Now we can prove the irreducibility of the highest $[0,\pm1]$-derivatives 
of $\rho|\cdot|^{\pm1}$-reduced irreducible representations.
\begin{prop}\label{010-1}
Let $\pi\in\Irr(G_n)$. 
Suppose that $\pi$ is $\rho|\cdot|^{-1}$-reduced (\resp $\rho|\cdot|^1$-reduced). 
Then the highest $[0,-1]$-derivative $D_{[0,-1]}^{(k)}(\pi)$ 
(\resp the highest $[0,1]$-derivative $D_{[0,1]}^{(k)}(\pi)$) is irreducible. 
Moreover, $\Delta_\rho[0,-1]^r \rtimes \pi$ (\resp $Z_\rho[0,1]^r \rtimes \pi$) is SI.
\end{prop}
\begin{proof}
We prove the assertions only for $[0,1]$. 
By the previous lemma, 
there exists an irreducible subrepresentation of $\pi_{[0,1]}$ of 
the highest $[0,1]$-derivative $D_{[0,1]}^{(k)}(\pi)$ such that 
\[
\Jac_{P_{2dk}}(\pi) \twoheadrightarrow Z_\rho[0,1]^k \boxtimes \pi_0, 
\]
or equivalently, 
\[
\pi \hookrightarrow Z_\rho[0,1]^k \rtimes \pi_0. 
\]
Since $\pi$ is $\rho|\cdot|^1$-reduced, so is $\pi_0$. 
Hence by Tadi{\'c}'s formula (Proposition \ref{eq:Tadic}) 
for $\left[\Jac_{P_{2dk}}(Z_\rho[0,1]^k \rtimes \pi_0)\right]$, 
we see that 
\[
D_{[0,1]}^{(k)}(Z_\rho[0,1]^k \rtimes \pi_0) = \pi_0. 
\]
Hence $0 \not= D_{[0,1]}^{(k)}(\pi) \subset \pi_0$ so that $D_{[0,1]}^{(k)}(\pi) = \pi_0$.
Moreover, it implies that $Z_\rho[0,1]^k \rtimes \pi_0$ is SI. 
\par

When $\pi'$ is an irreducible subrepresentation of $Z_\rho[0,1]^r \rtimes \pi$, 
we have $\pi' \subset \soc(Z_\rho[0,1]^{k+r} \rtimes \pi_0)$. 
In particular, $\pi'$ is unique and appears with multiplicity one in $\left[Z_\rho[0,1]^{k+r} \rtimes \pi_0\right]$, 
hence in $\left[Z_\rho[0,1]^r \rtimes \pi\right]$. 
Therefore, $Z_\rho[0,1]^r \rtimes \pi$ is SI.
\end{proof}

For simplicity, we set 
\[
S_{[0,1]}^{(r)}(\pi) = S_{Z_\rho[0,1]}^{(r)}(\pi) \coloneqq \soc(Z_\rho[0,1]^r \rtimes \pi)
\]
for an irreducible representation $\pi$ of $G_n$ which is $\rho|\cdot|^1$-reduced. 
\par

The highest $[0,-1]$-derivatives are easy in a special case. 
\begin{prop}\label{der-0-1}
Let $\pi = L(\Delta_{\rho_1}[x_1,y_1], \dots, \Delta_{\rho_r}[x_r,y_r]; \pi_\temp)$ 
be an irreducible representation of $G_n$. 
Suppose that $\pi$ is $\rho|\cdot|^z$-reduced  for all $z \not= 0$
and that there exists $i \in \{1, \dots, r\}$ such that $\rho_i \cong \rho$. 
Then $\min\{x_i \;|\; \rho_i \cong \rho\} = 0$, 
and the highest $[0,-1]$-derivative $D_{[0,-1]}^{(k)}(\pi)$ of $\pi$ is given by 
\[
D_{[0,-1]}^{(k)}(\pi) = L(\Delta_{\rho_1}[z_1,y_1], \dots, \Delta_{\rho_r}[z_r,y_r]; \pi_\temp)
\]
with
\[
z_i = \left\{
\begin{aligned}
&-2 \iif \rho_i \cong \rho, x_i = 0, \\
&x_i \other.
\end{aligned}
\right. 
\]
In particular, 
\[
k = |\{i \in \{1, \dots, r\} \;|\; \rho_i \cong \rho,\; x_i = 0\}| \geq 1.
\]
\end{prop}
\begin{proof}
With $x \coloneqq \min\{x_i \;|\; \rho_i \cong \rho\}$, we see that $\pi$ is not $\rho|\cdot|^x$-reduced. 
Hence we must have $x = 0$. 
Moreover, we note that if $\rho_i \cong \rho$ and $x_i = 0$, then $y_i \leq -1$ since $x_i+y_i < 0$. 
\par

Remark that $D_\rho^{(l)}(\pi_\temp)$ is tempered since $\rho$ is self-dual (\cite[Theorem 4.2 (1), (4)]{At1}), 
so that $D_\rho^{(l)}(\pi_\temp)$ is $\rho|\cdot|^{-1}$-reduced by Casselman's criterion (see e.g., \cite[Lemma 2.4]{K}). 
Hence by Lemma \ref{[01]-red}, with $k$ as in the statement, 
$D_{[0,-1]}^{(k)}(\pi)$ is the highest $[0,-1]$-derivative. 
\par

Set $\tau \coloneqq L(\Delta_{\rho_1}[x_1,y_1], \dots, \Delta_{\rho_r}[x_r,y_r])$.
Then $\pi \hookrightarrow \tau \rtimes \pi_\temp$. 
Since $\min\{x_i \;|\; \rho_i \cong \rho\} = 0$ and $y_i < 0$, 
we see that $\tau \hookrightarrow \Delta_{\rho}[0,-1]^k \times \tau'$ 
with $\tau' \coloneqq L(\Delta_{\rho_1}[z_1,y_1], \dots, \Delta_{\rho_r}[z_r,y_r])$.
Hence
\[
\pi \hookrightarrow \Delta_{\rho}[0,-1]^k \times \tau' \rtimes \pi_\temp. 
\]
By the Frobenius reciprocity, we have a nonzero map
\[
\Jac_{P_{2dk}}(\pi) \rightarrow \Delta_{\rho}[0,-1]^k \boxtimes (\tau' \rtimes \pi_\temp), 
\]
which must factor through a nonzero map
\[
\Delta_{\rho}[0,-1]^k \boxtimes D_{[0,-1]}^{(k)}(\pi) \rightarrow 
\Delta_{\rho}[0,-1]^k \boxtimes (\tau' \rtimes \pi_\temp). 
\]
Since $D_{[0,-1]}^{(k)}(\pi)$ is irreducible by Proposition \ref{010-1}, 
and since $\tau' \rtimes \pi_\temp$ is SI, 
we deduce that 
\[
D_{[0,-1]}^{(k)}(\pi) = \soc(\tau' \rtimes \pi_\temp).
\]
This completes the proof. 
\end{proof}

\subsection{Zelevinsky--Aubert duality and derivatives}\label{aubert}
We deduce the following compatibility between derivatives and duality. 
\begin{prop}
Let $\pi \in \Irr(G_n)$ and $\rho \in \Cusp(\GL_d(F))$. 
\begin{enumerate}
\item 
If $D_{\rho}^{(k)}(\pi)$ is the highest $\rho$-derivative, then 
\[
D_{\rho}^{(k)}(\pi){\widehat{\;}} = D_{\rho^\vee}^{(k)}(\hat\pi).
\]

\item 
If $\rho$ is self-dual, $\pi$ is $\rho|\cdot|^{-1}$-reduced and 
$D_{\Delta_\rho[0,-1]}^{(k)}(\pi)$ is the highest $\Delta_\rho[0,-1]$-derivative, 
then
\[
D_{\Delta_\rho[0,-1]}^{(k)}(\pi){\widehat{\;}} = D_{Z_\rho[0,1]}^{(k)}(\hat\pi).
\]
\end{enumerate}
\end{prop}
\begin{proof}
This is a consequence of the commutativity of the Jacquet functor with the duality, see \eqref{eq:Jacquet2}. 
\end{proof}

\section{The algorithm}\label{sec:algo}
Now we give an algorithm to compute the Zelevinsky--Aubert dual of an irreducible representation $\pi$. 
Thanks to Jantzen decomposition (see Paragraph \ref{sub:Jantzen}), 
we can reduce $\pi$ to the case where $\pi$ is either ugly or of good or bad parity. 
Then we proceed as follows:

\begin{alg}\label{alg}
Assume that we can compute $\hat\pi_0$ for all irreducible representations of $G_{n_0}$ for $n_0 < n$.
Let $\pi$ be an irreducible representation of $G_n$. 

\begin{enumerate}
\item
If there exists $\rho \in \Cusp^\GL$ such that $\rho$ is not self-dual 
and such that $D_{\rho}^{(k)}(\pi)$ is the highest $\rho$-derivative with $k \geq 1$, 
then
\[
\hat\pi = S_{\rho^\vee}^{(k)}\left( D_{\rho}^{(k)}(\pi){\widehat{\;}} \right).
\]

\item
Otherwise, and if $\pi$ is not tempered, 
then one can find $\rho \in \Cusp^\GL$ such that $\rho$ is self-dual and 
$D_{\Delta_\rho[0,-1]}^{(k)}(\pi)$ is the highest $\Delta_\rho[0,-1]$-derivative with $k \geq 1$. 
Then 
\[
\hat\pi = S_{Z_\rho[0,1]}^{(k)}\left( D_{\Delta_\rho[0,-1]}^{(k)}(\pi){\widehat{\;}} \right).
\]

\item
Otherwise, and if $\pi$ is tempered, 
then we can use an explicit formula for $\hat\pi$ (Proposition \ref{temp} below).
\end{enumerate}
\end{alg}

In order to run the algorithm we establish:
\begin{itemize}
\item
Explicit formulas for the highest $\rho$-derivative $D_{\rho}^{(k)}(\pi)$ and for the socle $S_{\rho}^{(k)}(\pi)$ 
for any $\rho \in \Cusp^\GL$ which is not self-dual. 
These are done in Proposition \ref{soc-negative} if $\rho$ is ugly or if the exponent of $\rho$ is negative, 
and in Theorem \ref{soc-positive} (\resp in Theorem \ref{bad-soc-positive}) if the exponent of $\rho$ is positive and 
if $\rho$ is in the good (\resp bad) case. 

\item
Explicit formulas for the $\Delta_\rho[0,-1]$-derivative $D_{\Delta_\rho[0,-1]}^{(k)}(\pi)$ 
and the socle $S_{Z_\rho[0,1]}^{(k)}(\pi)$ 
when $\rho$ is self-dual and $\pi$ is non-tempered and is $\rho|\cdot|^z$-reduced  for all $z \not= 0$. 
These are carried out in Proposition \ref{der-0-1} for the $\Delta_\rho[0,-1]$-derivative 
and in Theorem \ref{der-01} for the socle, respectively.

\item
an explicit formula for $\hat\pi$ 
when $\pi$ is tempered such that $\pi$ is 
$\rho|\cdot|^z$-reduced for all $z \not= 0$. This is done in Proposition \ref{temp}.
\end{itemize}

In the rest of the paper, we will prove all these formulas.

\section{The endoscopic classification}\label{endoscopy}
In Paragraphs \ref{sub:posgood} and \ref{subpos01} below, 
we will give explicit formulas for several derivatives and socles in the good parity case. 
In these formulas, certain special irreducible representations $\pi_A$ play an important and mysterious role.  
These special representations $\pi_A$ are \emph{of Arthur type}, 
and the mystery comes from Arthur's theory of the endoscopic classification \cite{Ar}. 
In this section, we review his theory. 

\subsection{$A$-parameters}
We denote by $W_F$ the Weil group of $F$.
A homomorphism 
\[
\psi \colon W_F \times \SL_2(\C) \times \SL_2(\C) \rightarrow \GL_n(\C)
\]
is called an \emph{$A$-parameter for $\GL_n(F)$} 
if 
\begin{itemize}
\item
$\psi(\Frob) \in \GL_n(\C)$ is semisimple and all its eigenvalues have absolute value $1$, 
where $\Frob$ is a fixed (geometric) Frobenius element;
\item
$\psi|W_F$ is smooth, i.e., has an open kernel; 
\item
$\psi|\SL_2(\C) \times \SL_2(\C)$ is algebraic.
\end{itemize}
The local Langlands correspondence for $\GL_d(F)$ asserts that 
there is a canonical bijection between
the set of irreducible unitary supercuspidal representations of $\GL_d(F)$
and 
the set of irreducible $d$-dimensional representations of $W_F$ of bounded image. 
We identify these two sets, and use the symbol $\rho$ for their elements. 
\par

Any irreducible representation of $W_F \times \SL_2(\C) \times \SL_2(\C)$ 
is of the form $\rho \boxtimes S_a \boxtimes S_b$, 
where $S_a$ is the unique irreducible algebraic representation of $\SL_2(\C)$ of dimension $a$.
We shortly write $\rho \boxtimes S_a = \rho \boxtimes S_a \boxtimes S_1$ 
and $\rho = \rho \boxtimes S_1 \boxtimes S_1$. 
For an $A$-parameter $\psi$, 
the multiplicity of $\rho \boxtimes S_a \boxtimes S_b$ in $\psi$
is denoted by $m_\psi(\rho \boxtimes S_a \boxtimes S_b)$.
When $\psi = \oplus_{i \in I} \rho_i \boxtimes S_{a_i} \boxtimes S_{b_i}$ is an $A$-parameter of $\GL_n(F)$, 
we define $\tau_\psi$ by the product of Speh representations (see Paragraph \ref{rep0})
\[
\tau_\psi \coloneqq \bigtimes_{i \in I} 
L\left( 
\Delta_{\rho_i}\left[ \half{a_i-b_i}, -\half{a_i+b_i}+1 \right], \dots, \Delta_{\rho_i}\left[\half{a_i+b_i}-1, -\half{a_i-b_i} \right]
\right).
\]
\par

Now we consider a split odd special orthogonal group $\SO_{2n+1}(F)$ 
or a symplectic group $\Sp_{2n}(F)$. 
We call $\psi$ an \emph{$A$-parameter for $\SO_{2n+1}(F)$} 
if it is an $A$-parameter for $\GL_{2n}(F)$ of symplectic type, i.e., 
\[
\psi \colon W_F \times \SL_2(\C) \times \SL_2(\C) \rightarrow \Sp_{2n}(\C). 
\]
Similarly, 
$\psi$ is called an \emph{$A$-parameter for $\Sp_{2n}(F)$} 
if it is an $A$-parameter for $\GL_{2n+1}(F)$ of orthogonal type with the trivial determinant, i.e.,
\[
\psi \colon W_F \times \SL_2(\C) \times \SL_2(\C) \rightarrow \SO_{2n+1}(\C). 
\]
For $G_n = \SO_{2n+1}(F)$ (\resp $G_n = \Sp_{2n}(F)$), 
we let $\Psi(G_n)$ be the set of $\widehat{G_n}$-conjugacy classes of $A$-parameters for $G_n$, 
where $\widehat{G_n} = \Sp_{2n}(\C)$ (\resp $\widehat{G_n} = \SO_{2n+1}(\C)$).
We say that 
\begin{itemize}
\item
$\psi \in \Psi(G_n)$ is \emph{tempered} 
if the restriction of $\psi$ to the second $\SL_2(\C)$ is trivial; 
\item
$\psi \in \Psi(G_n)$ is \emph{of good parity} 
if $\psi$ is a sum of irreducible self-dual representations of the same type as $\psi$.
\end{itemize}
We denote by $\Psi_\temp(G_n) \coloneqq \Phi_\temp(G_n)$ (\resp $\Psi_\gp(G_n)$)
the subset of $\Psi(G)$ consisting of tempered $A$-parameters 
(\resp $A$-parameters of good parity).
Also, we put $\Phi_\gp(G_n) \coloneqq \Phi_\temp(G_n) \cap \Psi_\gp(G_n)$. 
Set $\Psi_*(G) \coloneqq \cup_{n \geq 0}\Psi_*(G_n)$ and $\Phi_*(G) \coloneqq \cup_{n \geq 0}\Phi_*(G_n)$
for $* \in \{\emptyset, \temp, \gp\}$. 
\par

For $\psi \in \Psi(G)$, a component group $\Sc_\psi$ is defined. 
We recall the definition only when $\psi \in \Psi_\gp(G)$. 
Hence we can write $\psi = \oplus_{i=1}^{r} \psi_i$, 
where $\psi_i$ is an irreducible self-dual representation of the same type as $\psi$. 
We define an \emph{enhanced component group} $\AA_\psi$ as 
\[
\AA_\psi \coloneqq \bigoplus_{i=1}^r (\Z/2\Z)\alpha_{\psi_i}. 
\]
Namely, $\AA_\psi$ is a free $\Z/2\Z$-module of rank $r$
with a basis $\{\alpha_{\psi_i}\}$ associated with the irreducible components $\{\psi_i\}$. 
Define the \emph{component group} $\Sc_\psi$ as
the quotient of $\AA_\psi$ by the subgroup generated by the elements 
\begin{itemize}
\item
$z_\psi \coloneqq \sum_{i=1}^r \alpha_{\psi_i}$; and
\item
$\alpha_{\psi_i} + \alpha_{\psi_{i'}}$ such that $\psi_i \cong \psi_{i'}$.  
\end{itemize}
Let $\widehat{\Sc_\psi}$ and $\widehat{\AA_\psi}$ be the Pontryagin duals of $\Sc_\psi$ and $\AA_\psi$, 
respectively. 
Via the canonical surjection $\AA_\psi \twoheadrightarrow \Sc_\psi$, 
we may regard $\widehat{\Sc_\psi}$ as a subgroup of $\widehat{\AA_\psi}$. 
For $\eta \in \widehat{\AA_\psi}$, we write $\eta(\alpha_{\psi_i}) = \eta(\psi_i)$. 
\par

Let $\Irr_\unit(G_n)$ (\resp $\Irr_\temp(G_n)$) 
be the set of equivalence classes of irreducible unitary (\resp tempered) representations of $G_n$.
For $\psi \in \Psi(G_n)$, 
Arthur \cite[Theorem 2.2.1]{Ar} defined a multiset $\Pi_\psi$ over $\Irr_\unit(G_n)$, 
which is called the \emph{$A$-packet} for $G_n$ associated with $\psi$. 
It satisfies the following properties:

\begin{itemize}
\item
$\Pi_\psi$ is actually a (multiplicity-free) subset of $\Irr_\unit(G_n)$ (M{\oe}glin \cite{Moe3}). 

\item
There exists a map $\Pi_\psi \rightarrow \widehat{\Sc_\psi}$, $\pi \mapsto \pair{\cdot,\pi}_\psi$. 
If $\phi \in \Phi_\temp(G)$, it is a bijection. 
When $\pi \in \Pi_\phi$ corresponds to $\eta \in \widehat{\Sc_\phi}$, 
we write $\pi = \pi(\phi, \eta)$.

\item
There is a canonical decomposition into a disjoint union 
\[
\Irr_\temp(G_n) = \bigsqcup_{\phi \in \Phi_\temp(G_n)}\Pi_\phi.
\]

\item
If $\psi = \psi_1 \oplus \psi_0 \oplus \psi_1^\vee$ for some irreducible representation $\psi_1$, 
then there exists a canonical injection $\Sc_{\psi_0} \hookrightarrow \Sc_{\psi}$, and 
\[
\tau_{\psi_1} \rtimes \pi_0 \cong 
\bigoplus_{\substack{\pi \in \Pi_\psi \\ \pair{\cdot, \pi}_\psi|\Sc_{\psi_0} = \pair{\cdot, \pi_0}_{\psi_0}}}
\pi.
\]
for every $\pi_0 \in \Pi_{\psi_0}$ (see \cite[Proposition 2.4.3]{Ar}).
\end{itemize}
\par

\begin{rem}\label{rem:good}
Let $\rho \in \Cusp^\GL$ be unitary and $x \geq 0$ be a real number. 
Then the following are equivalent: 
\begin{enumerate}
\item
For any $\pi(\phi, \eta)$ with $\phi \in \Phi_\gp(G)$ and $\eta \in \widehat{\Sc_\phi}$, 
there exists $m \in \Z$ such that $\rho|\cdot|^{x+m} \rtimes \pi(\phi, \eta)$ is reducible. 

\item
For some $\pi(\phi, \eta)$ with $\phi \in \Phi_\gp(G)$ and $\eta \in \widehat{\Sc_\phi}$, 
there exists $m \in \Z$ such that $\rho|\cdot|^{x+m} \rtimes \pi(\phi, \eta)$ is reducible. 

\item
$x \in (1/2)\Z$ and $\rho \boxtimes S_{2x+1}$ is self-dual of the same type as elements of $\Phi_\gp(G)$, i.e., 
\begin{itemize}
\item
$x \in \Z$ and $\rho$ is self-dual of the same type as elements of $\Phi_\gp(G)$; or 
\item
$x \in (1/2)\Z \setminus \Z$ and $\rho$ is self-dual of the opposite type to elements of $\Phi_\gp(G)$. 
\end{itemize}
\end{enumerate}
This follows, for example, from \cite[Th{\'e}or{\`e}me (i)]{MW} and \cite[Theorem 4.7]{J-irrr}.
In particular, $\rho|\cdot|^x$ is good in the sense of Definition \ref{def:good} if and only if 
$\rho \boxtimes S_{2x+1}$ is self-dual of the same type as elements of $\Phi_\gp(G)$. 
Also, an irreducible representation 
$\pi = L(\Delta_{\rho_1}[x_1,y_1], \dots, \Delta_{\rho_r}[x_r,y_r]; \pi_\temp)$
is of good parity
if and only if $\pi_\temp = \pi(\phi, \eta)$ with $\phi \in \Phi_\gp(G)$, 
and $\rho_i \boxtimes S_{2|x_i|+1}$ is self-dual of the same type as $\phi$ for all $i = 1, \dots, r$.
\end{rem}

\subsection{A special example}
Now, we consider a special $A$-parameter of the form
\[
\psi = \phi \oplus (\rho \boxtimes S_{2x} \boxtimes S_2)^{t}
\]
for $t \geq 1$, 
$\phi \in \Phi_\gp(G)$, and 
$x \in (1/2)\Z$ with $x > 0$ such that $\rho \boxtimes S_{2x+1}$ is self-dual of the same type as $\phi$.  
\par

For $l \in \Z/2\Z$ and for $\eta$ in a certain subset $\widehat{\Sc_{\psi,l}}$ in $\widehat{\Sc_\psi}$ (depending on $l$), 
we will define $\pi(\psi, l, \eta)$ as follows. 
When $l = 1$, 
we set 
$\widehat{\Sc_{\psi,1}} \coloneqq \widehat{\Sc_\phi} 
= \{ \eta \in \widehat{\Sc_\psi} \;|\; \eta(\rho \boxtimes S_{2x} \boxtimes S_2) = 1\}$, 
and
\[
\pi(\psi, 1, \eta) \coloneqq L(\Delta_\rho[x-1,-x]^t; \pi(\phi, \eta)). 
\]
When $l = 0$ and $x \geq 1$, we set $\widehat{\Sc_{\psi,0}}$ to be the subset of $\widehat{\Sc_\psi}$
consisting of $\eta$ satisfying 
\begin{itemize}
\item
$\eta(\rho \boxtimes S_{2x} \boxtimes S_2) = \eta(\rho \boxtimes S_{2x-1})$ 
if $\rho \boxtimes S_{2x-1} \subset \phi$; 
\item
$\eta(\rho \boxtimes S_{2x} \boxtimes S_2) = (-1)^{t}\eta(\rho \boxtimes S_{2x+1})$ 
if $\rho \boxtimes S_{2x+1} \subset \phi$; 
\item
$\eta(z_\phi) = (-1)^t$. 
\end{itemize}
When $l = 0$ and $x = 1/2$, 
we set $\widehat{\Sc_{\psi,0}}$ to be the subset of $\widehat{\Sc_\psi}$
consisting of $\eta$ satisfying 
\begin{itemize}
\item
$\eta(\rho \boxtimes S_{1} \boxtimes S_2) = -1$; 
\item
$\eta(\rho \boxtimes S_{2}) = (-1)^{t}$ if $\rho \boxtimes S_{2} \subset \phi$; 
\item
$\eta(z_\phi) = (-1)^t$. 
\end{itemize}
For $\eta \in \widehat{\Sc_{\psi,0}}$, we define 
\[
\pi(\psi, 0, \eta) \coloneqq L(\Delta_\rho[x-1,-x]^{t-1}; \pi(\phi + \rho \boxtimes (S_{2x-1}+S_{2x+1}), \eta)). 
\]
Here, we regard $\eta$ as a character of the component group of $\phi + \rho \boxtimes (S_{2x-1}+S_{2x+1})$
by setting 
\[
\left\{
\begin{aligned}
&\eta(\rho \boxtimes S_{2x-1}) = (-1)^t\eta(\rho \boxtimes S_{2x+1}) = \eta(\rho \boxtimes S_{2x} \boxtimes S_2) 
\iif x \geq 1, \\
&\eta(\rho \boxtimes S_{2}) = (-1)^{t} \iif x = 1/2.
\end{aligned}
\right. 
\] 
\par

By specifying M{\oe}glin's construction of $\Pi_\psi$, we have the following. 
\begin{prop}\label{moe}
Let $\psi = \phi \oplus (\rho \boxtimes S_{2x} \boxtimes S_2)^{t} \in \Psi_\gp(G)$ with $t \geq 1$. 
Then 
\[
\Pi_\psi = \left\{ \pi(\psi, l, \eta) \;\middle|\; l \in \Z/2\Z, \; \eta \in \widehat{\Sc_{\psi,l}} \right\}. 
\]
Moreover, 
the map $\Pi_\psi \rightarrow \widehat{\Sc_\psi}$ is given by $\pair{\cdot,\pi(\psi, l ,\eta)}_\psi = \varepsilon_{l,\eta}$, 
where
\begin{align*}
\varepsilon_{l,\eta}(\rho \boxtimes S_d) &= \eta(\rho \boxtimes S_d), \\
\varepsilon_{l,\eta}(\rho \boxtimes S_{2x} \boxtimes S_2) 
&= 
\left\{
\begin{aligned}
&(-1)^{l-1} \iif x \geq 1, \\
&\eta(\rho \boxtimes S_{1} \boxtimes S_2) \iif x = 1/2.
\end{aligned}
\right. 
\end{align*}
\end{prop}
\begin{proof}
The $A$-packet $\Pi_\psi$ was constructed by M{\oe}glin explicitly. 
See \cite[\S 8]{X2} for details. 
For $x \geq 1$, its construction was computed in \cite[Proposition 3.13]{At2}. 
The same calculation can be applied to $x = 1/2$. 
By \cite[Corollary 8.10]{X2}, the map $\Pi_\psi \rightarrow \widehat{\Sc_\psi}$ is given by 
$\pair{\cdot,\pi(\psi, l ,\eta)}_\psi = \varepsilon_{l,\eta} \cdot \epsilon_\psi^{M/W}$
for some character $\epsilon_\psi^{M/W} \in \widehat{\Sc_{\psi}}$.
By definition (\cite[Definitions 5.2, 5.5, 8.1]{X2}), 
one can easily see that $\epsilon_\psi^{M/W} = \1$ in our case. 
\end{proof}

Using this description, we obtain the formula for the highest $\rho|\cdot|^x$-derivatives and socles. 
\begin{thm}\label{der-special}
Fix $\phi \in \Phi_\gp(G)$ and write $m = m_{\phi}(\rho \boxtimes S_{2x+1})$ 
and $m' = m_{\phi}(\rho \boxtimes S_{2x-1})$. 
Consider $\psi = \phi \oplus (\rho \boxtimes S_{2x} \boxtimes S_2)^t \in \Psi_\gp(G)$ with $t \geq 0$. 
Let $\pi(\psi, l, \eta) \in \Pi_{\psi}$ be such that 
$\eta(\rho \boxtimes S_{2x-1})\eta(\rho \boxtimes S_{2x+1}) = (-1)^t$ 
if $mm' \not= 0$.
Here, if $x=1/2$, 
we formally understand that $m'=1$ and $\eta(\rho \boxtimes S_0) = 1$.
Let $s$ be a non-negative integer such that $s=0$ if $x=1/2$.
Then the highest $\rho|\cdot|^x$-derivative of $\soc((\rho|\cdot|^{-x})^{s} \rtimes \pi(\psi, l ,\eta))$ is given by
\begin{align*}
&D_{\rho|\cdot|^x}^{(m+\max\{s-m',0\})}\left(\soc\left((\rho|\cdot|^{-x})^{s} \rtimes \pi(\psi, l ,\eta)\right)\right) 
\\&= 
\soc\left( 
(\rho|\cdot|^{-x})^{\min\{s,m'\}} 
\rtimes \pi(\psi - (\rho \boxtimes S_{2x+1})^m + (\rho \boxtimes S_{2x-1})^m, l+m, \eta) 
\right),
\end{align*}
where we set $\eta(\rho \boxtimes S_{2x-1}) = (-1)^t \eta(\rho \boxtimes S_{2x+1})$. 
In particular, 
\begin{align*}
&S_{\rho|\cdot|^x}^{(1)}\left(\soc\left((\rho|\cdot|^{-x})^{s} \rtimes \pi(\psi, l ,\eta)\right)\right)
\\&= 
\left\{
\begin{aligned}
&\soc\left((\rho|\cdot|^{-x})^{s} \rtimes 
\pi(\psi - \rho \boxtimes S_{2x-1} + \rho \boxtimes S_{2x+1}, l-1, \eta)
\right)
\iif s < m', \\
&
\soc\left((\rho|\cdot|^{-x})^{s+1} \rtimes \pi(\psi, l ,\eta)\right)
\iif s \geq m', 
\end{aligned}
\right. 
\end{align*}
where we set $\eta(\rho \boxtimes S_{2x+1}) = (-1)^t \eta(\rho \boxtimes S_{2x-1})$. 
\end{thm}
\begin{proof}
When $x \geq 1$ (\resp $x=1/2$), 
the formula for the highest $\rho|\cdot|^x$-derivatives was obtained in \cite[Theorem 4.1]{At2} 
(\resp in \cite[Theorem 3.3]{J-dual}). 
It implies the formula for socles. 
\end{proof}

\subsection{Zelevinsky--Aubert duals of certain tempered representations}\label{sec.temp}
The initial step of our algorithm to compute the Zelevinsky--Aubert duals (Algorithm \ref{alg} (3))
is to compute $\hat\pi$ for tempered $\pi$ 
such that $\pi$ is $\rho'$-reduced for every non-self-dual $\rho'\in \Cusp^\GL$. 
If $\pi = \pi(\phi, \eta)$ for $\phi \in \Phi_\gp(G)$, then $\pi$ satisfies this condition if and only if: 
\begin{itemize}
\item[$(\ast)$]
if $\rho \boxtimes S_{d} \subset \phi$ with $d \geq 2$, 
then $m_\phi(\rho \boxtimes S_{d}) =1$, $\rho \boxtimes S_{d-2} \subset \phi$ 
and $\eta(\rho \boxtimes S_{d}) \not= \eta(\rho \boxtimes S_{d-2})$. 
\end{itemize}
See \cite[Theorem 4.2]{At1}. 
Here, we formally understand that $\rho \boxtimes S_0 \subset \phi$ and $\eta(\rho \boxtimes S_0) = +1$
if $\rho$ is self-dual of the opposite type to $\phi$. 

\begin{prop}\label{temp}
Let $\pi = \pi(\phi, \eta)$ with $\phi \in \Phi_\gp(G)$. 
Assume that $\pi$ satisfies the above condition $(\ast)$. 
Write
\[
\{\rho \;|\; m_\phi(\rho) >0,\; m_\phi(\rho) \equiv 0 \bmod 2\} = \{\rho_1, \dots, \rho_r\}
\]
and set
\[
y_i \coloneqq \max\left\{\half{d_i-1} \;\middle|\; \rho_i \boxtimes S_{d_i} \subset \phi \right\}. 
\]
Suppose that $y_1 \geq \dots \geq y_t > 0 = y_{t+1} = \dots = y_r$.
Then 
\[
\hat\pi = L(\Delta_{\rho_1}[0,-y_1], \dots, \Delta_{\rho_t}[0,-y_t]; \pi(\phi', \eta')), 
\] 
where
\[
\phi' = \phi - \bigoplus_{i=1}^t \rho_i \boxtimes (S_1+S_{2y_i+1}) 
\]
and 
\[
\eta'(\rho \boxtimes S_{d}) = 
\left\{
\begin{aligned}
&-\eta(\rho \boxtimes S_d) \iif \rho \in \{\rho_1, \dots, \rho_r\}, \\
&\eta(\rho \boxtimes S_d) \other. 
\end{aligned}
\right. 
\]
\end{prop}
\begin{proof}
Set
\[
\{\rho \;|\; m_\phi(\rho) >0,\; m_\phi(\rho) \equiv 1 \bmod 2\} = \{\rho'_1, \dots, \rho'_{r'}\}. 
\]
Write $m_\phi(\rho_i) = 2k_i > 0$ and $m_{\phi}(\rho'_j) = 2k'_j+1$. 
Then by \cite[Theorem 4.2]{At1}, we have
\[
\left(\circ_{j=1}^{r'}D_{\rho_j'}^{(k'_j)}\right) \circ 
\left(\circ_{i=1}^{r}D_{\rho_i|\cdot|^{y_i}}^{(1)} \circ \dots \circ D_{\rho_i|\cdot|^1}^{(1)} \circ D_{\rho_i}^{(k_i)}\right)(\pi) \not= 0.
\]
It is $\pi(\phi'', \eta'')$ up to multiplicity, where 
\[
\phi'' = \phi - \left(\bigoplus_{j=1}^{r'}{\rho'_j}^{2k_j'}\right) 
- \left(\bigoplus_{i=1}^r\rho_i \boxtimes (S_1^{2k_i-1}+S_{2y_i+1})\right)
\]
and 
\[
\eta''(\rho \boxtimes S_{d}) = 
\left\{
\begin{aligned}
&-\eta(\rho \boxtimes S_d) \iif \rho \in \{\rho_1, \dots, \rho_t\}, \\
&\eta(\rho \boxtimes S_d) \iif \rho \not\in \{\rho_1, \dots, \rho_r\}.
\end{aligned}
\right. 
\]
Note that $\rho_i \not\subset \phi''$ for $i > t$. 
In particular, $\pi(\phi'', \eta'')$ is supercuspidal. 
By \cite[Theorem 2.13]{At2}, 
with $\phi'$ as in the statement, we have 
\[
\hat\pi = L(\Delta_{\rho_1}[0,-y_1], \dots, \Delta_{\rho_t}[0,-y_t]; \pi(\phi', \eta'))
\]
for some $\eta' \in \AA_{\phi'}$ 
such that $\eta'' = \eta'|\AA_{\phi''}$ via the canonical inclusion $\AA_{\phi''} \hookrightarrow \AA_{\phi'}$.
Since $\Sc_{\phi'}$ is generated by $\Sc_{\phi''}$ and the image of $\{\alpha_{\rho_i} \;|\; i > t\}$, 
the remaining task is to determine $\eta'(\rho_{i_0})$ for $i_0 > t$. 
To do this, by replacing $\pi$ with 
\[
\left(\circ_{j=1}^{r'}D_{\rho_j'}^{(k'_j)}\right) \circ \left(\circ_{\substack{1 \leq i \leq r \\ i \not= i_0}}
D_{\rho_i|\cdot|^{y_i}}^{(1)} \circ \dots \circ D_{\rho_i|\cdot|^1}^{(1)} \circ D_{\rho_i}^{(k_i)}\right)(\pi), 
\]
we may assume that $\pi \subset \rho^k \rtimes \sigma$
with $\sigma$ supercuspidal such that $\rho \rtimes \sigma$ is semisimple of length two. 
If we write $\rho \rtimes \sigma = \pi_+ \oplus \pi_-$, 
then $\rho^{k-1} \rtimes \pi_{\pm}$ is irreducible and 
its Zelevinsky--Aubert dual is given by $\rho^{k-1} \rtimes \hat\pi_{\pm}$. 
By \cite[Corollaire 1.10]{Au}, we know that $\hat\pi_{\pm} = \pi_{\mp}$. 
Hence we see that $\eta'(\rho_{i_0}) = -\eta(\rho_{i_0})$, as desired.
\end{proof}

If $\pi$ is tempered, of $\rho$-bad parity, and $\rho|\cdot|^z$-reduced  for all $z \not= 0$, 
then $\pi$ must be of the form $\pi = \rho^m \rtimes \sigma$ for some $m \geq 0$ and $\sigma$ supercuspidal. 
In particular, we have $\hat\pi = \pi$.
Similarly, if $\pi$ is tempered, ugly and $\rho'$-reduced for all non-self-dual $\rho'\in \Cusp^\GL$, 
then $\pi$ must be supercuspidal so that $\hat\pi = \pi$.

\section{Best matching functions, the ugly and the negative case}\label{sub:best}
To give formulas for derivatives and socles, 
following \cite[\S 5.3]{LM}, we introduce the notion of the best matching functions. We then use these functions to explicate the ugly and the negative case.
\par

\subsection{Best matching functions}
Let $A$ and $B$ be totally ordered finite sets 
with respect to $\geq_A$ and $\geq_B$, respectively.
For $a \in A$, write $A_{>a} \coloneqq \{a' \in A \;|\; a' >_A a\}$. 
We consider a relation $\rel$ between $B$ and $A$ such that
\begin{align*}
&\forall a_1 \geq_A a_2 \in A,\;
\forall b_1 \geq_B b_2 \in B, \\
&b_1 \rel a_1 \;\&\; b_2 \rel a_1 \;\&\; b_2 \rel a_2 
\implies b_1 \rel a_2.
\end{align*}
We call such a relation \emph{traversable}.
In this case, 
we define a subset $A^0$ of $A$ and an injective map $f \colon A^0 \rightarrow B$ recursively by
\begin{align*}
\text{$a \in A^0 \iff \exists b \in B\setminus f(A^0 \cap A_{>a})$ such that $b \rel a$}
\\
\text{in which case $f(a) \coloneqq \min\{b \in B \setminus f(A^0 \cap A_{>a}) \;|\; b \rel a\}$}.
\end{align*}
Set $B^0 \coloneqq f(A^0)$ to be the image of $f$.
We call the bijection $f \colon A^0 \rightarrow B^0$ 
the \textit{best matching function} between $A$ and $B$.
By \cite[Lemma 5.7]{LM}, the domain $A^0$ is equal to $A$
if and only if Hall's criterion is satisfied, i.e., 
for any subset $A' \subset A$, we have
\[
|\{b \in B \;|\; \text{$b \rel a$ for some $a \in A'$}\}| \geq |A'|.
\]
When one of $A$ or $B$ is the empty set, note that we have $A^0 = B^0 = \emptyset$. 
We set $A^c=A\setminus A^0$ and $B^c = B \setminus B^0$. 
\par

\subsection{Derivatives and socles in the ugly and in the negative case}\label{negative}
Fix $\rho \in \Cusp^\GL$ and $x \in \R$. 
In this subsection, we give explicit formulas using the best matching functions 
for the highest $\rho|\cdot|^x$-derivatives $D_{\rho|\cdot|^x}^{(k)}(\pi)$ and
the socles $S_{\rho|\cdot|^x}^{(1)}(\pi) = \soc(\rho|\cdot|^x \rtimes \pi)$ 
in the case where $\rho|\cdot|^x$ is ugly, or $\rho$ is self-dual and $x$ is negative. 
\par

Let $\pi \in \Irr(G_n)$. 
By Remark \ref{rem:ugly} and by the Langlands classification, 
we can write
$\pi = \soc(L(\Delta_{\rho_1}[x_1,y_1], \dots, \Delta_{\rho_r}[x_r,y_r]) \rtimes \pi_\temp)$
where:
\begin{itemize}
\item
if $\rho|\cdot|^x$ is ugly, 
then $\rho_i = \rho$ for all $i = 1, \dots, r$, $x_1+y_1 \leq \dots \leq x_r+y_r$ and $\pi_\temp = \sigma$ is supercuspidal; 
\item
if $\rho$ is self-dual and $x$ is negative, 
then $x_1+y_1 \leq \dots \leq x_r+y_r < 0$, and $\pi_\temp$ is tempered. 
\end{itemize}
To unify notation, let us call $(\Delta_{\rho_1}[x_1,y_1], \dots, \Delta_{\rho_r}[x_r,y_r]; \pi_\temp)$
the \emph{inducing data}.
\par

Define an ordered set $A_{\rho|\cdot|^x}$ by
\[
A_{\rho|\cdot|^x} \coloneqq \{ i \in \{1, \dots, r\} \;|\; \rho_i \cong \rho, \; x_i = x\}
\]
with 
\[
a \geq a' \iff y_a \geq y_{a'}. 
\]
We define a relation $\rel$ between $A_{\rho|\cdot|^x}$ and $A_{\rho|\cdot|^{x-1}}$ by 
\[
A_{\rho|\cdot|^x} \ni a' \rel a \in A_{\rho|\cdot|^{x-1}} \iff y_{a'} > y_a. 
\]
Namely, $a' \rel a$ if and only if $L(\Delta_\rho[x_a,y_a], \Delta_\rho[x_{a'},y_{a'}])$ is a ladder representation.
Note that this relation is traversable. 
Let $f \colon A_{\rho|\cdot|^{x-1}}^0 \rightarrow A_{\rho|\cdot|^{x}}^0$ be the best matching function. 
In the next proposition, 
we obtain explicit formulas for the highest $\rho|\cdot|^x$-derivative $D_{\rho|\cdot|^x}^{(k)}(\pi)$
and for the socle $S_{\rho|\cdot|^x}^{(1)}(\pi)$. 

\begin{prop}\label{soc-negative}
Suppose $\rho|\cdot|^x$ is ugly or $\rho$ is self-dual and $x$ is negative. 
With notation as above, 
the highest $\rho|\cdot|^x$-derivative $D_{\rho|\cdot|^x}^{(k)}(\pi)$ is the unique irreducible subrepresentation of 
$L(\Delta_{\rho_1}[x'_1,y_1], \dots, \Delta_{\rho_r}[x'_r,y_r]) \rtimes \pi_\temp$, 
where
\[
x_i' = \left\{
\begin{aligned}
&x-1	\iif i \in A_{\rho|\cdot|^{x}}^c,\\
&x_i	\other. 
\end{aligned}
\right.  
\]
In particular, $k = |A_{\rho|\cdot|^{x}}^c|$.
Moreover: 
\begin{enumerate}
\item[(a)]
If $A_{\rho|\cdot|^{x-1}}^c\neq \emptyset$, then
the inducing data of $S_{\rho|\cdot|^x}^{(1)}(\pi)$ can be obtained from those of $\pi$ 
by replacing $x_a = x-1$ with $x$, 
where $a$ is the minimum element of $A_{\rho|\cdot|^{x-1}}^c$. 

\item[(b)]
If $A_{\rho|\cdot|^{x-1}}^c = \emptyset$, 
then the inducing data of $S_{\rho|\cdot|^x}^{(1)}(\pi)$ can be obtained from those of $\pi$ 
by inserting $\rho|\cdot|^x = \Delta_\rho[x,x]$.
\end{enumerate}
\end{prop}
\begin{proof}
Since $\rho|\cdot|^x$ is ugly or $\rho$ is self-dual and $x$ is negative, we have 
\begin{align*}
D_{\rho|\cdot|^x}^{(k)}(\pi) 
&= \soc\left(
L_{\rho|\cdot|^x}^{(k)}(L(\Delta_{\rho_1}[x_1,y_1], \dots, \Delta_{\rho_r}[x_r,y_r])) \rtimes \pi_\temp
\right), \\
S_{\rho|\cdot|^x}^{(1)}(\pi) 
&= \soc\left(
\soc(\rho|\cdot|^x \times L(\Delta_{\rho_1}[x_1,y_1], \dots, \Delta_{\rho_r}[x_r,y_r])) \rtimes \pi_\temp
\right). 
\end{align*}
Therefore, the proposition is essentially a problem for general linear groups,
which was done in \cite[Theorem 5.11]{LM}.
\end{proof}

\section{Explicit formulas for derivatives and socles: The positive case}\label{sec:pos}
In this section, we give explicit formulas for the highest derivatives
and for the socles of several parabolically induced representations in the positive case. 
The main results are Theorem \ref{soc-positive} 
where we describe derivatives and socles in the good parity case, 
and Theorem \ref{bad-soc-positive} where the bad parity case is treated. 
In Corollary \ref{irr_cusp} we deduce a result on irreducibility of certain parabolic inductions.

We fix in all this section $\rho \in \Cusp^\GL$ self-dual, and $x \in (1/2)\Z$ with $x > 0$.
\par

\subsection{Good parity case}\label{sub:posgood}
In this subsection, 
we assume that $\pi\in\Irr(G_n)$ is of good parity, 
and that $\rho \boxtimes S_{2x+1}$ is self-dual of the same type as elements in $\Phi_\gp(G)$. 
Write $\pi = L(\Delta_{\rho_1}[x_1,y_1], \dots, \Delta_{\rho_{r'}}[x_{r'},y_{r'}]; \pi(\phi, \eta))$
as a Langlands subrepresentation so that $x_1+y_1 \leq \dots \leq x_{r'}+y_{r'} <0$ and $\phi \in \Phi_\gp(G)$. 
Set 
\[
t = |\{i \in \{1, \dots, r'\}\;|\; \Delta_{\rho_i}[x_i,y_i] \cong \Delta_\rho[x-1,-x]\}|
\]
and $r = r'-t$.
Then we can rewrite 
\[
\pi = \soc\left(
L(\Delta_{\rho_1}[x_1,y_1], \dots, \Delta_{\rho_r}[x_r,y_r]) \rtimes \pi_A 
\right),
\]
where we set $\pi_A \coloneqq L(\Delta_\rho[x-1,-x]^t; \pi(\phi, \eta))$. 
\par

If $m_\phi(\rho \boxtimes S_{2x+1}) \not= 0$, $m_\phi(\rho \boxtimes S_{2x-1}) \not= 0$
and $\eta(\rho \boxtimes S_{2x+1})\eta(\rho \boxtimes S_{2x-1}) = (-1)^{t+1}$, 
set 
\[
\psi \coloneqq \phi - \rho \boxtimes (S_{2x+1}+S_{2x-1}) + (\rho \boxtimes S_{2x} \boxtimes S_2)^{t+1}
\]
and $l \coloneqq 0$. 
Otherwise, set $\psi \coloneqq \phi + (\rho \boxtimes S_{2x} \boxtimes S_2)^t$ and $l \coloneqq 1$. 
Then $\pi_A = \pi(\psi, l, \eta) \in \Pi_\psi$ by Proposition \ref{moe}.
Set $m \coloneqq m_\psi(\rho \boxtimes S_{2x+1})$ and $m' \coloneqq m_\psi(\rho \boxtimes S_{2x-1})$.
Then the highest $\rho|\cdot|^x$-derivative of 
$\soc((\rho|\cdot|^{-x})^s \rtimes \pi_A)$ is described in Theorem \ref{der-special}.
\par

Note that $x_i \geq y_i$ for all $i = 1, \dots, r$.
Define ordered sets 
\begin{align*}
A_{\rho|\cdot|^{x}} &\coloneqq \{i \in \{1,\dots,r\} \;|\; \rho_i \cong \rho,\; x_i = x\}, \\ 
B_{\rho|\cdot|^x} &\coloneqq \{i \in \{1,\dots,r\} \;|\; \rho_i \cong \rho,\; y_i = -x\}
\end{align*}
with 
\begin{align*}
a \geq a' \iff y_a \geq y_{a'}
&\quad\text{for $a,a' \in A_{\rho|\cdot|^{x}}$}, \\
b \geq b' \iff x_b \leq x_{b'}
&\quad\text{for $b,b' \in B_{\rho|\cdot|^x}$}. 
\end{align*}
Notice that any two of $A_{\rho|\cdot|^{x-1}}, A_{\rho|\cdot|^{x}}, B_{\rho|\cdot|^{x-1}}, B_{\rho|\cdot|^{x}}$
have no intersection. 
Define relations $\rel$ between $A_{\rho|\cdot|^{x}}$ and $A_{\rho|\cdot|^{x-1}}$, 
and between $B_{\rho|\cdot|^{x}}$ and $B_{\rho|\cdot|^{x-1}}$ by
\begin{align*}
&A_{\rho|\cdot|^{x}} \ni a' \rel a \in A_{\rho|\cdot|^{x-1}} \iff y_{a'} > y_a, \\
&B_{\rho|\cdot|^{x}} \ni b' \rel b \in B_{\rho|\cdot|^{x-1}} \iff x_{b'} < x_b, 
\end{align*}
respectively.
Note that these relations are traversable.
Let $f \colon A_{\rho|\cdot|^{x-1}}^0 \rightarrow A_{\rho|\cdot|^{x}}^0$ 
and $g \colon B_{\rho|\cdot|^{x-1}}^0 \rightarrow B_{\rho|\cdot|^{x}}^0$ be the best matching functions. 
Write $B_{\rho|\cdot|^{x}}^c= \{i_1, \dots, i_{s}\}$ with $i_1 < \dots < i_{s}$.
We notice that $s > 0$ only if $x > 1$.

\begin{thm}\label{soc-positive}
Notation is as above. 
Suppose that $x > 0$, $x \in (1/2)\Z$ 
and that $\rho \boxtimes S_{2x+1}$ is self-dual of the same type as $\phi$. 
Then the highest $\rho|\cdot|^x$-derivative $D_{\rho|\cdot|^x}^{(k)}(\pi)$ is the unique irreducible subrepresentation of 
$L(\Delta_{\rho_1}[x'_1,y'_1], \dots, \Delta_{\rho_r}[x'_r,y'_r]) \rtimes \pi_A'$,
where
\begin{align*}
x_i' &= \left\{
\begin{aligned}
&x-1	\iif i \in A_{\rho|\cdot|^{x}}^c,\\
&x_i	\other, 
\end{aligned}
\right.  \\
y_i' &= \left\{
\begin{aligned}
&-(x-1)	
\iif \text{$i=i_j$, $j > m'+\max\{|A_{\rho|\cdot|^{x-1}}^c|-m,0\}$}, \\
&y_i	\other, 
\end{aligned}
\right.  
\end{align*}
and $\pi_A' = \pi(\psi', l', \eta)$ with 
\[
\psi' = \psi 
- (\rho \boxtimes S_{2x+1})^{\max\{m-|A_{\rho|\cdot|^{x-1}}^c|,0\}} 
+ (\rho \boxtimes S_{2x-1})^{\max\{m-|A_{\rho|\cdot|^{x-1}}^c|,0\}}
\]
and 
\[
l' = l + \max\{m-|A_{\rho|\cdot|^{x-1}}^c|,0\}.
\]
In particular, 
\[
k = |A_{\rho|\cdot|^{x}}^c| 
+ \max\left\{
m+\max\{|B_{\rho|\cdot|^{x}}^c|-m',0\}
-|A_{\rho|\cdot|^{x-1}}^c|,0
\right\}.
\]
Moreover: 
\begin{enumerate}
\item[(a)]
If 
$m+\max\{|B_{\rho|\cdot|^{x}}^c|-m',0\} 
< |A_{\rho|\cdot|^{x-1}}^c|$, 
then the Langlands data of $S_{\rho|\cdot|^x}^{(1)}(\pi)$ can be obtained from those of $\pi$ 
by replacing $x_a = x-1$ with $x$, where $a$ is the minimum element of 
$A_{\rho|\cdot|^{x-1}}^c$. 

\item[(b)]
If $|B_{\rho|\cdot|^{x}}^c| < m'$ 
and $m \geq |A_{\rho|\cdot|^{x-1}}^c|$, 
the Langlands data of $S_{\rho|\cdot|^x}^{(1)}(\pi)$ can be obtained from those of $\pi$ 
by replacing $\pi_A = \pi(\psi, l, \eta)$ with 
\[
S_{\rho|\cdot|^x}^{(1)}(\pi_A) = \pi(\psi - (\rho \boxtimes S_{2x-1}) + (\rho \boxtimes S_{2x+1}), l-1, \eta). 
\]

\item[(c)]
If $|B_{\rho|\cdot|^{x}}^c| \geq m'$, 
$m+|B_{\rho|\cdot|^{x}}^c|-m' \geq |A_{\rho|\cdot|^{x-1}}^c|$ 
and $B_{\rho|\cdot|^{x-1}}^c \not= \emptyset$, 
the Langlands data of $S_{\rho|\cdot|^x}^{(1)}(\pi)$ can be obtained from those of $\pi$ 
by replacing $y_b = -(x-1)$ with $-x$, 
where $b$ is the minimum element of $B_{\rho|\cdot|^{x-1}}^c$. 

\item[(d)]
If $|B_{\rho|\cdot|^{x}}^c| \geq m'$, 
$m+|B_{\rho|\cdot|^{x}}^c|-m' \geq |A_{\rho|\cdot|^{x-1}}^c|$ 
and $B_{\rho|\cdot|^{x-1}}^c = \emptyset$, 
then the Langlands data of $S_{\rho|\cdot|^x}^{(1)}(\pi)$ can be obtained from those of $\pi$ 
by inserting $\rho|\cdot|^{-x} = \Delta_{\rho}[-x,-x]$.
\end{enumerate}
\end{thm}
\begin{proof}
To obtain the formula for the highest derivative, 
we use Jantzen's algorithm \cite[\S 3.3]{J-dual} together with \cite[Theorem 5.11]{LM} and Theorem \ref{der-special}.
\begin{enumerate}
\item
Recall that
\[
\pi = \soc\left( L(\Delta_{\rho_1}[x_1,y_1], \dots, \Delta_{\rho_r}[x_r,y_r]) \rtimes \pi_A \right)
\]
with $\pi_A = L(\Delta_\rho[x-1,-x]^t; \pi(\phi, \eta))$ and 
$\Delta_{\rho_i}[x_i,y_i] \not\cong \Delta_\rho[x-1,-x]$ for all $i = 1, \dots,r$.

\item
By \cite[Theorem 5.11]{LM}, 
we can compute the highest right $\rho|\cdot|^{-x}$-derivative 
\[
R_{\rho|\cdot|^{-x}}^{(s)}(L(\Delta_{\rho_1}[x_1,y_1], \dots, \Delta_{\rho_r}[x_r,y_r]))
= L(\Delta_{\rho_1}[x_1,y_1''], \dots, \Delta_{\rho_r}[x_r,y_r'']), 
\]
where 
\[
y_i'' = \left\{
\begin{aligned}
&-(x-1) \iif i \in B_{\rho|\cdot|^{x}}^c, \\
&y_i \other.
\end{aligned}
\right. 
\]
In particular, $s = |B_{\rho|\cdot|^{x}}^c|$.
Jantzen's Claim 1 in \cite[\S 3.3]{J-dual} says that
\[
\pi = \soc\left( L(\Delta_{\rho_1}[x_1,y_1''], \dots, \Delta_{\rho_r}[x_r,y_r'']) \rtimes \pi_1 \right)
\]
with $\pi_1 \coloneqq \soc((\rho|\cdot|^{-x})^s \rtimes \pi_A)$.

\item
By Theorem \ref{der-special}, 
the highest $\rho|\cdot|^x$-derivative $\pi_2 \coloneqq D_{\rho|\cdot|^x}^{(k_1)}(\pi_1)$ of $\pi_1$ is 
\[
\pi_2 = \soc\left( (\rho|\cdot|^{-x})^{\min\{s,m'\}} \rtimes 
\pi(\psi - (\rho \boxtimes S_{2x+1})^m + (\rho \boxtimes S_{2x-1})^m, l+m,\eta) \right)
\]
with $k_1 = m + \max\{s-m', 0\}$. 
Jantzen's Claim 2 in \cite[\S 3.3]{J-dual} says that
\[
\pi = 
\soc\left( 
L(\Delta_{\rho_1}[x_1,y_1''], \dots, \Delta_{\rho_r}[x_r,y_r''], (\rho|\cdot|^x)^{k_1}) \rtimes \pi_2 
\right).
\]

\item
We will apply \cite[Theorem 5.11]{LM} 
to compute the highest left $\rho|\cdot|^x$-derivative of
$L(\Delta_{\rho_1}[x_1,y_1''], \dots, \Delta_{\rho_r}[x_r,y_r''], (\rho|\cdot|^x)^{k_1})$. 
To do this, we have to replace $A_{\rho|\cdot|^x}$ with $A_{\rho|\cdot|^x} \cup \{r+1, \dots, r+k_1\}$, 
where we set $\Delta_{\rho_i}[x_i,y_i] = \rho|\cdot|^x$ for $i = r+1, \dots, r+k_1$.
Note that any $a' \in \{r+1, \dots, r+k_1\}$ is bigger than any element of $A_{\rho|\cdot|^x}$
with respect to the order of $A_{\rho|\cdot|^x} \cup \{r+1, \dots, r+k_1\}$, 
and $a' \rel a$ for every $a \in A_{\rho|\cdot|^{x-1}}$. 
Hence the image of the resulting best matching function is 
\[
A_{\rho|\cdot|^x}^0 \cup 
\left\{ r+i \;\middle|\; 1 \leq i \leq \min\{k_1, |A_{\rho|\cdot|^{x-1}}^c|\} \right\}.
\]
Therefore, with $k_2 = \min\{k_1, |A_{\rho|\cdot|^{x-1}}^c|\}$
and $k = |A_{\rho|\cdot|^{x}}^c| + k_1-k_2$, 
the highest left $\rho|\cdot|^x$-derivative is 
\begin{align*}
&L_{\rho|\cdot|^x}^{(k)}
\left(L(\Delta_{\rho_1}[x_1,y_1''], \dots, \Delta_{\rho_r}[x_r,y_r''], (\rho|\cdot|^x)^{k_1})\right)
\\&=
L(\Delta_{\rho_1}[x_1',y_1''], \dots, \Delta_{\rho_r}[x_r',y_r''], (\rho|\cdot|^x)^{k_2}), 
\end{align*}
where $x_i'$ is as in the statement of this theorem.
Then the highest $\rho|\cdot|^x$-derivative of $\pi$ is 
\[
D_{\rho|\cdot|^x}^{(k)}(\pi)
= \soc\left(
L(\Delta_{\rho_1}[x_1',y_1''], \dots, \Delta_{\rho_r}[x_r',y_r''], (\rho|\cdot|^x)^{k_2}) \rtimes \pi_2 
\right).
\]

\item
Jantzen's Claim 3 in \cite[\S 3.3]{J-dual} says that
\[
D_{\rho|\cdot|^x}^{(k)}(\pi)
= \soc\left(
L(\Delta_{\rho_1}[x_1',y_1''], \dots, \Delta_{\rho_r}[x_r',y_r'']) \rtimes S_{\rho|\cdot|^x}^{(k_2)}(\pi_2)
\right). 
\]
By Theorem \ref{der-special}, we have
\[
S_{\rho|\cdot|^x}^{(k_2)}(\pi_2) = \soc((\rho|\cdot|^{-x})^{s'} \rtimes \pi_A'), 
\]
where 
$\pi_A'$ is as in the statement of this theorem, 
and $s' = \min\{s,m'\} + \max\{k_2-m,0\}$. 
Note that $s' \leq s$. 

\item
Finally, note that
\begin{itemize}
\item
if $s'=s$, then $m'+\max\{|A_{\rho|\cdot|^{x-1}}^c|-m,0\} \geq s$, 
so that $y_i' = y_i$ for all $i=1, \dots, r$; 
\item
if $s' < s$, then $s > m'$ and $k_1 = m+s-m' > k_2 = |A_{\rho|\cdot|^{x-1}}^c|$
so that 
$s' = m'+\max\{|A_{\rho|\cdot|^{x-1}}^c|-m,0\}$. 
\end{itemize}
By \cite[Theorem 5.11]{LM}, we have
\begin{align*}
&\soc\left(
L(\Delta_{\rho_1}[x_1',y_1''], \dots, \Delta_{\rho_r}[x_r',y_r'']) \times (\rho|\cdot|^{-x})^{s'}
\right)
\\&= L(\Delta_{\rho_1}[x_1',y_1'], \dots, \Delta_{\rho_r}[x_r',y_r']),  
\end{align*}
where $y_i'$ is as in the statement of this theorem.
Jantzen's Claim 4 in \cite[\S 3.3]{J-dual} says that
\[
D_{\rho|\cdot|^x}^{(k)}(\pi)
= \soc\left(
L(\Delta_{\rho_1}[x_1',y_1'], \dots, \Delta_{\rho_r}[x_r',y_r']) \rtimes \pi_A'
\right).
\]
This gives the Langlands data of $D_{\rho|\cdot|^x}^{(k)}(\pi)$.
\end{enumerate}

Recall that $S_{\rho|\cdot|^x}^{(1)}(\pi)$ is an irreducible representation determined by the relation
\[
D_{\rho|\cdot|^x}^{(k+1)} \left( S_{\rho|\cdot|^x}^{(1)}(\pi) \right) = D_{\rho|\cdot|^x}^{(k)}(\pi). 
\]
One can easily check this equation for the representations given in (a), (b), (c) and (d).
\end{proof}

As an application of Proposition \ref{soc-negative} and Theorem \ref{soc-positive}, 
we have a combinatorial irreducibility criterion for $\rho|\cdot|^x \rtimes \pi$ as follows.
\begin{cor}\label{irr_cusp}
Notation is as above. 
Suppose that $x >0$, $x \in (1/2)\Z$ 
and that $\rho \boxtimes S_{2x+1}$ is self-dual of the same type as $\phi$. 
Then the parabolically induced representation $\rho|\cdot|^x \rtimes \pi$ is irreducible 
if and only if all of the following conditions hold. 
\begin{itemize}
\item
$A_{\rho|\cdot|^{-x-1}}^c = \emptyset$; 
\item
$|B_{\rho|\cdot|^{x}}^c| \geq m_\psi(\rho \boxtimes S_{2x-1})$; 
\item
$m_\psi(\rho \boxtimes S_{2x+1})
+|B_{\rho|\cdot|^{x}}^c|
-m_\psi(\rho \boxtimes S_{2x-1}) 
\geq |A_{\rho|\cdot|^{x-1}}^c|$;
\item
$B_{\rho|\cdot|^{x-1}}^c = \emptyset$.
\end{itemize}
\end{cor}
\begin{proof}
Note that $\rho|\cdot|^x \rtimes \pi$ is irreducible 
if and only if $S_{\rho|\cdot|^x}^{(1)}(\pi) \cong S_{\rho|\cdot|^{-x}}^{(1)}(\pi)$. 
By Proposition \ref{soc-negative} and Theorem \ref{soc-positive}, 
this is equivalent to the case where 
the Langlands data of $S_{\rho|\cdot|^{-x}}^{(1)}(\pi)$ and $S_{\rho|\cdot|^x}^{(1)}(\pi)$ 
are obtained from those of $\pi$ by inserting $\rho|\cdot|^{-x}$.
\end{proof}

As a special case, when $\pi = \pi(\phi, \eta)$ is tempered, 
since $A_{\rho^\vee|\cdot|^{-x-1}}, A_{\rho|\cdot|^{x-1}}, A_{\rho|\cdot|^x}, B_{\rho|\cdot|^{x-1}}, B_{\rho|\cdot|^x}$
are all the empty set, 
we see that $\rho|\cdot|^x \rtimes \pi$ if and only if $m_\psi(\rho \boxtimes S_{2x-1}) = 0$, 
which is equivalent that 
\begin{itemize}
\item
$\phi \not\supset \rho \boxtimes S_{2x-1}$; or
\item
$m_\phi(\rho \boxtimes S_{2x-1}) = 1$, $m_\phi(\rho \boxtimes S_{2x+1}) > 0$ 
and $\eta(\rho \boxtimes S_{2x-1}) \not= \eta(\rho \boxtimes S_{2x+1})$. 
\end{itemize}
This special case was already known by Jantzen \cite[Theorem 4.7]{J-irrr}. 

\subsection{Bad parity case}\label{sub:posbad}
We treat now the bad parity case. 
Namely, we assume that $\rho \boxtimes S_{2x+1}$ is self-dual of the opposite type to elements in $\Phi_\gp(G)$, 
and we take $\pi \in \Irr(G_n)$ such that $\scusp(\pi) \subset \Z_{\rho|\cdot|^x} \cup \{\sigma\}$ 
for some $\sigma \in \Cusp^G$. 
\par

Remark that Jantzen's algorithm \cite[\S 3.3]{J-dual} to compute the highest $\rho|\cdot|^x$-derivatives
can be applied to the bad parity case. 
According to this algorithm (see (2) in the proof of Theorem \ref{soc-positive}), 
we have to treat a $\rho|\cdot|^x$-bad representation of the form
\[
\pi_1 = L((\rho|\cdot|^{-x})^s, \Delta_\rho[x-1,-x]^t; \pi(\phi, \eta))
\]
with $\phi \in \Phi_\temp(G_n)$ and $s,t \geq 0$. 
Here, we may assume that $s = 0$ if $x=1/2$ since $\rho|\cdot|^{-1/2} = \Delta_\rho[-1/2,-1/2]$.
By the assumption of the bad parity, if we write $\sigma = \pi(\phi_\sigma, \eta_\sigma)$, 
then $\phi = \phi_\sigma \oplus (\oplus_{i=1}^r (\rho \boxtimes S_{2x_i+1})^{m_i})$
with $x_i \in x + \Z$ so that $\Sc_\phi \cong \Sc_{\phi_\sigma}$, and $\eta = \eta_\sigma$.
Moreover, the multiplicity $m_i$ is even for all $i$.
The following is an extension of \cite[Propositions 8.5, 8.6]{J-dual}.

\begin{prop}\label{bad-special}
Notation is as above. 
Here, when $x=1/2$, we assume that $s=0$.
Set $m \coloneqq m_\phi(\rho \boxtimes S_{2x+1})$ and $m' \coloneqq m_\phi(\rho \boxtimes S_{2x-1})$, 
both of which are even.
Take $\kappa \in \{0,1\}$ such that $t \equiv \kappa \bmod 2$. 
Then the highest $\rho|\cdot|^x$-derivative $D_{\rho|\cdot|^x}^{(k)}(\pi_1)$ is equal to
\[
L(
(\rho|\cdot|^{-x})^{\min\{s,m'+\kappa\}}, \Delta_\rho[x-1,-x]^{t-\kappa};
\pi(\phi - (\rho \boxtimes S_{2x+1})^m + (\rho \boxtimes S_{2x-1})^{m+2\kappa}, \eta)
)
\]
with $k = m+\kappa+\max\{s-m'-\kappa,0\}$.
\end{prop}
\begin{proof}
If we write $\pi_0 \coloneqq \pi(\phi - (\rho \boxtimes S_{2x+1})^m - (\rho \boxtimes S_{2x-1})^{m'}, \eta)$, 
then 
\[
\pi(\phi, \eta) = \Delta_\rho[x-1,-(x-1)]^{\half{m'}} \times \Delta_\rho[x,-x]^{\half{m}} \rtimes \pi_0
\]
is an irreducible induction.
Moreover, 
\[
\Delta_\rho[x-1,-x] \times \Delta_\rho[x-1,-(x-1)]^{\half{m'}} \times \Delta_\rho[x,-x]^{\half{m}} \rtimes \pi_0
\] 
is always irreducible by \cite[Th{\'e}or{\`e}me (i)]{MW}. 
Also, any subquotient of $\Delta_\rho[x-1,-x] \times \Delta_\rho[x,-(x-1)]$ is 
$\Delta_\rho[x-1,-(x-1)] \times \Delta_\rho[x,-x]$ or $L_0 \coloneqq L(\Delta_\rho[x-1,-x], \Delta_\rho[x,-(x-1)])$, 
both of which commute with all of $\Delta_\rho[x-1,-(x-1)]$, $\Delta_\rho[x,-x]$ and $\Delta_\rho[x-1,-x]$.
\par

First we assume that $t$ is even. 
By considering the Langlands data, we have
\begin{align*}
&\soc\left(
\Delta_\rho[x-1,-x]^t \times \Delta_\rho[x-1,-(x-1)]^{\half{m'}} \times \Delta_\rho[x,-x]^{\half{m}}
\rtimes \pi_0
\right)
\\&\hookrightarrow
L_0^{\half{t}}
\times \Delta_\rho[x-1,-(x-1)]^{\half{m'}} \times \Delta_\rho[x,-x]^{\half{m}} \rtimes \pi_0
\\&\hookrightarrow 
\Delta_\rho[x-1,-x]^t \times \Delta_\rho[x-1,-(x-1)]^{\half{m'}} \times \Delta_\rho[x,-x]^{\half{m}} \rtimes \pi_0. 
\end{align*}
Since the middle induced representation is unitary and 
since the last induced representation is a standard module so that it is SI, 
we see that the first inclusion map is an isomorphism.
In particular, $\pi_1$ is equal to the socle of
\begin{align*}
&(\rho|\cdot|^{-x})^s \times
L_0^{\half{t}}
\times \Delta_\rho[x-1,-(x-1)]^{\half{m'}} \times \Delta_\rho[x,-x]^{\half{m}} \rtimes \pi_0
\\&\cong 
L_0^{\half{t}} \times
(\rho|\cdot|^{-x})^s \times \Delta_\rho[x-1,-(x-1)]^{\half{m'}} \times \Delta_\rho[x,-x]^{\half{m}} \rtimes \pi_0.
\end{align*}
Therefore, we may replace $(\rho|\cdot|^{-x})^s \times \Delta_\rho[x-1,-(x-1)]^{\half{m'}}$ with 
\[
\tag{$\ast$}
(\rho|\cdot|^{-x})^{\max\{s-\half{m'},0\}} 
\times L_1^{\min\{s,\half{m'}\}} 
\times \Delta_\rho[x-1,-(x-1)]^{\max\{\half{m'}-s,0\}}, 
\]
where $L_1 \coloneqq L(\rho|\cdot|^{-x}, \Delta_\rho[x-1,-(x-1)])$. 
Moreover,
since $\rho|\cdot|^{-x} \times \Delta_\rho[x,-x]^{\half{m}} \rtimes \pi_0$ is irreducible by \cite[Th{\'e}or{\`e}me (i)]{MW}, 
if $s \geq \half{m'}$, then we may replace $(\ast)$ with 
\[
\tag{$\ast\ast$}
(\rho|\cdot|^{-x})^{\max\{s-m',0\}} \times L_2^{\min\{s-\half{m'}, \half{m'}\}} \times  L_1^{\max\{m'-s,0\}}, 
\]
where $L_2 \coloneqq L(\rho|\cdot|^{-x}, \Delta_\rho[x-1,-(x-1)], \rho|\cdot|^x)$. 
Note that if $x \geq 1$, by \cite[Proposition 5.15 (3)]{LM}, 
the ladder representations $L_0$, $L_1$ and $L_2$ commute with all of 
\[
\Delta_\rho[x,-x],\; \Delta_\rho[x-1,-x],\; \Delta_\rho[x,-(x-1)],\; \Delta_\rho[x-1,-(x-1)].
\]
Therefore, with 
\[
k = m+\max\{s-m',0\}, 
\]
the $\rho|\cdot|^x$-derivative $D_{\rho|\cdot|^x}^{(k)}(\pi)$ is the highest and is a subrepresentation of 
\[
\left\{
\begin{aligned}
&L_0^{\half{t}} \times L_1^{s} \times \Delta_\rho[x-1,-(x-1)]^{\half{m'}-s+\half{m}} \rtimes \pi_0
\iif s \leq \half{m'}, \\
&L_0^{\half{t}} \times L_2^{s-\half{m'}} \times  L_1^{m'-s} \times \Delta_\rho[x-1,-(x-1)]^{\half{m}} \rtimes \pi_0
\iif \half{m'} < s \leq m', \\
&
L_0^{\half{t}} \times L_2^{\half{m'}} \times \Delta_\rho[x-1,-(x-1)]^{\half{m}} \rtimes \pi_0
\iif s > m'.
\end{aligned}
\right.
\]
Since $L_2 \times L_1 \cong L_1 \times L_2$ by \cite[Corollary 6.2]{LM} 
and since $L_1 \rtimes \sigma$ is irreducible by \cite[Theorem 1.2]{LT}, 
this representation is a subrepresentation of 
\[
\left\{
\begin{aligned}
&(\rho|\cdot|^{-x})^s \times \Delta_\rho[x-1,-x]^t \times \Delta_\rho[x-1,-(x-1)]^{\half{m'+m}} \rtimes \pi_0
\iif s \leq m', \\
&(\rho|\cdot|^{-x})^{m'} \times \Delta_\rho[x-1,-x]^t \times \Delta_\rho[x-1,-(x-1)]^{\half{m'+m}} \rtimes \pi_0
\iif s > m'.
\end{aligned}
\right.
\]
Since 
$\Delta_\rho[x-1,-(x-1)]^{\half{m'+m}} \rtimes \pi_0 
= \pi(\phi - (\rho \boxtimes S_{2x+1})^m + (\rho \boxtimes S_{2x-1})^{m}, \eta)$,
we obtain the case where $t$ is even. 
\par

Next, we assume that $t$ is odd.
By considering the Langlands data, we have
\begin{align*}
&\soc\left(
\Delta_\rho[x-1,-x]^t \times \Delta_\rho[x-1,-(x-1)]^{\half{m'}} \times \Delta_\rho[x,-x]^{\half{m}}
\rtimes \pi_0
\right)
\\&\hookrightarrow
L_0^{\half{t-1}} \times \Delta_\rho[x-1,-x]
\times \Delta_\rho[x-1,-(x-1)]^{\half{m'}} \times \Delta_\rho[x,-x]^{\half{m}} \rtimes \pi_0
\\&\cong
L_0^{\half{t-1}} \times \Delta_\rho[x,-(x-1)]
\times \Delta_\rho[x-1,-(x-1)]^{\half{m'}} \times \Delta_\rho[x,-x]^{\half{m}} \rtimes \pi_0. 
\end{align*}
Note that the middle induced representation is SI since it is a subrepresentation of a standard module. 
On the other hand, by taking the MVW-functor and the contragredient functor, 
we see that the unique irreducible subrepresentation of the middle induced representation
is also an irreducible quotient of the last induced representation. 
By the last isomorphism, this means that
$L_0^{\half{t-1}} \times \Delta_\rho[x,-(x-1)]
\times \Delta_\rho[x-1,-(x-1)]^{\half{m'}} \times \Delta_\rho[x,-x]^{\half{m}} \rtimes \pi_0$
is irreducible. 
Therefore, by the same argument as the case where $t$ is even, 
with $k =  m + 1 + \max\{s-m'-1,0\}$, the $\rho|\cdot|^x$-derivative $D_{\rho|\cdot|^x}^{(k)}(\pi)$ is highest 
and is a subrepresentation of 
\[
\left\{
\begin{aligned}
&(\rho|\cdot|^{-x})^s \times \Delta_\rho[x-1,-x]^{t-1} 
\times \Delta_\rho[x-1,-(x-1)]^{\half{m'+m}+1} \rtimes \pi_0
\iif s \leq m'+1, \\
&(\rho|\cdot|^{-x})^{m'+1} \times \Delta_\rho[x-1,-x]^{t-1} 
\times \Delta_\rho[x-1,-(x-1)]^{\half{m'+m}+1} \rtimes \pi_0
\iif s > m'+1.
\end{aligned}
\right.
\]
Since 
$\Delta_\rho[x-1,-(x-1)]^{\half{m'+m}+1} \rtimes \pi_0 
= \pi(\phi - (\rho \boxtimes S_{2x+1})^m + (\rho \boxtimes S_{2x-1})^{m+2}, \eta)$,
we obtain the case where $t$ is odd. 
\end{proof}

Now we consider the general case. 
Let $\pi = L(\Delta_\rho[x_1,y_1], \dots, \Delta_\rho[x_{r'},y_{r'}]; \pi(\phi, \eta))$
with $x_1+y_1 \leq \dots \leq x_{r'}+y_{r'} < 0$ and $\phi \in \Phi_\temp(G)$. 
If we define $t, r \geq 0$ with $t+r = r'$ as in \S \ref{sub:posgood}, one can rewrite 
\[
\pi = \soc \left(L(\Delta_{\rho}[x_1,y_1], \dots, \Delta_{\rho}[x_r,y_r]) \rtimes \pi_A\right), 
\]
where 
\begin{itemize}
\item
$x_1+y_1 \leq \dots \leq x_r+y_r < 0$; 
\item
$\pi_A \coloneqq L(\Delta_\rho[x-1,-x]^t; \pi(\phi, \eta))$; 
\item
$[x_i,y_i] \not= [x-1,-x]$ for all $i = 1, \dots,r$.
\end{itemize}
Set $m \coloneqq m_\phi(\rho \boxtimes S_{2x+1})$ and $m' \coloneqq m_\phi(\rho \boxtimes S_{2x-1})$, 
both of which are even.
Take $\kappa \in \{0,1\}$ such that $t \equiv \kappa \bmod 2$.  
\par

Define 
\begin{align*}
A_{\rho|\cdot|^{x}} &\coloneqq \{i \in \{1,\dots,r\} \;|\; x_i = x\}, \\ 
B_{\rho|\cdot|^x} &\coloneqq \{i \in \{1,\dots,r\} \;|\; y_i = -x\}. 
\end{align*}
As in in the previous paragraph, we regard $A_{\rho|\cdot|^{x}}$ and $A_{\rho|\cdot|^{x-1}}$
(\resp $B_{\rho|\cdot|^{x}}$ and $B_{\rho|\cdot|^{x-1}}$)
as ordered sets, and take the traversal relation $\rel$. 
Let $f \colon A_{\rho|\cdot|^{x-1}}^0 \rightarrow A_{\rho|\cdot|^{x}}^0$ 
(\resp $g \colon B_{\rho|\cdot|^{x-1}}^0 \rightarrow B_{\rho|\cdot|^{x}}^0$) be the best matching function. 
Write $B_{\rho|\cdot|^{x}}^c = \{i_1, \dots, i_{s}\}$ with $i_1 < \dots < i_{s}$.
Note that $s > 0$ only if $x > 1$.

\begin{thm}\label{bad-soc-positive}
Notation is as above. 
Suppose that $x > 0$, $x \in (1/2)\Z$ 
and that $\rho \boxtimes S_{2x+1}$ is self-dual of the opposite type to elements in $\Phi_\gp(G)$. 
Then the highest $\rho|\cdot|^x$-derivative $D_{\rho|\cdot|^x}^{(k)}(\pi)$ is the unique irreducible subrepresentation of 
$L(\Delta_{\rho_1}[x'_1,y'_1], \dots, \Delta_{\rho_r}[x'_r,y'_r]) \rtimes \pi_A'$,
where
\begin{align*}
x_i' &= \left\{
\begin{aligned}
&x-1	\iif i \in A_{\rho|\cdot|^{x}}^c,\\
&x_i	\other, 
\end{aligned}
\right.  \\
y_i' &= \left\{
\begin{aligned}
&-(x-1)	
\iif \text{$i=i_j$, 
$j > m'+\kappa+\max\{|A_{\rho|\cdot|^{x-1}}^c|-m-\kappa,0\}$}, \\
&y_i	\other, 
\end{aligned}
\right.  
\end{align*}
and 
\begin{itemize}
\item
if $m+\kappa \leq |A_{\rho|\cdot|^{x-1}}^c|$, 
then $\pi_A' = \pi_A$; 
\item
if $m+\kappa > |A_{\rho|\cdot|^{x-1}}^c|$, 
then 
\[
\pi_A' = 
\left\{
\begin{aligned}
&L \left(\Delta_\rho[x-1,-x]^{t-\kappa}; 
\pi(\phi - (\rho \boxtimes S_{2x+1})^{m-v} + (\rho \boxtimes S_{2x-1})^{m-v+2\kappa}, \eta)
\right),\\
&L\left( \Delta_\rho[x-1,-x]^{t-\kappa+1};
\pi(\phi - (\rho \boxtimes S_{2x+1})^{m-v+1} + (\rho \boxtimes S_{2x-1})^{m-v-1+2\kappa}, \eta)
\right)
\end{aligned}
\right. 
\] 
according to $v = |A_{\rho|\cdot|^{x-1}}^c|$ is even or odd. 
\end{itemize}
In particular, 
\[
k = |A_{\rho|\cdot|^{x}}^c| 
+ \max\left\{
m+\kappa+\max\{|B_{\rho|\cdot|^{x}}^c|-m'-\kappa,0\}
-|A_{\rho|\cdot|^{x-1}}^c|,0
\right\}.
\]
Moreover: 
\begin{enumerate}
\item[(a)]
If $m+\kappa+\max\{|B_{\rho|\cdot|^{x}}^c|-m'-\kappa,0\} 
< |A_{\rho|\cdot|^{x-1}}^c|$, 
then the Langlands data of $S_{\rho|\cdot|^x}^{(1)}(\pi)$ can be obtained from those of $\pi$ 
by replacing $x_a = x-1$ with $x$, where $a$ is the minimum element of 
$A_{\rho|\cdot|^{x-1}}^c$. 

\item[(b)]
If $|B_{\rho|\cdot|^{x}}^c| < m'+\kappa$ 
and $m+\kappa \geq |A_{\rho|\cdot|^{x-1}}^c|$, 
the Langlands data of $S_{\rho|\cdot|^x}^{(1)}(\pi)$ can be obtained from those of $\pi$ 
by replacing $\pi_A$ with 
\[
S_{\rho|\cdot|^x}^{(1)}(\pi_A) = 
\left\{\begin{aligned}
&L\left(
\Delta_\rho[x-1,-x]^{t+1};
\pi(\phi - (\rho \boxtimes S_{2x-1})^2, \eta)
\right)
\iif \kappa = 0, \\
&L\left(
\Delta_\rho[x-1,-x]^{t-1}; 
\pi(\phi + (\rho \boxtimes S_{2x+1})^2, \eta)
\right)
\iif \kappa = 1. 
\end{aligned}
\right.
\]

\item[(c)]
If $|B_{\rho|\cdot|^{x}}^c| \geq m'+\kappa$, 
$m+|B_{\rho|\cdot|^{x}}^c|-m' 
\geq |A_{\rho|\cdot|^{x-1}}^c|$ 
and $B_{\rho|\cdot|^{x-1}}^c \not= \emptyset$, 
the Langlands data of $S_{\rho|\cdot|^x}^{(1)}(\pi)$ can be obtained from those of $\pi$ 
by replacing $y_b = -(x-1)$ with $-x$, 
where $b$ is the minimum element of $B_{\rho|\cdot|^{x-1}}^c$. 

\item[(d)]
If $|B_{\rho|\cdot|^{x}}^c| \geq m'+\kappa$, 
$m+|B_{\rho|\cdot|^{x}}^c|-m' 
\geq |A_{\rho|\cdot|^{x-1}}^c|$ 
and $B_{\rho|\cdot|^{x-1}}^c = \emptyset$, 
then the Langlands data of $S_{\rho|\cdot|^x}^{(1)}(\pi)$ can be obtained from those of $\pi$ 
by inserting $\rho|\cdot|^{-x} = \Delta_{\rho}[-x,-x]$.
\end{enumerate}
\end{thm}
\begin{proof}
By a similar argument to Theorem \ref{soc-positive}, 
we obtain the assertions by applying Jantzen's algorithm \cite[\S 3.3]{J-dual} 
together with \cite[Theorem 5.11]{LM} and Proposition \ref{bad-special}. 
\end{proof}

As a consequence, one can obtain an analogous criterion to Corollary \ref{irr_cusp}
for the irreducibility of $\rho|\cdot|^x \rtimes \pi$. 
We leave the details to the reader. 

\section{Explicit formulas for derivatives and socles: A non-cuspidal case}\label{sec01}
Fix $\rho \in \Cusp^\GL$ self-dual. 
In this section, we consider $\pi \in \Irr(G_n)$ of good or $\rho$-bad parity satisfying that:
\begin{itemize}
\item[(a)]
$\pi$ is $\rho|\cdot|^1$-reduced; and 
\item[(b)]
$\pi$ is $\rho|\cdot|^z$-reduced for all $z < 0$. 
\end{itemize}
\par

Recall that if an irreducible representation $\pi$ is $\rho|\cdot|^1$-reduced, 
Proposition \ref{010-1} says that $Z_\rho[0,1]^k \rtimes \pi$ is SI. 
In this subsection, we determine the highest $[0,1]$-derivative $\pi' = D_{[0,1]}^{(k)}(\pi)$ of $\pi$, 
and we show how to recover the Langlands data of $\pi$ in terms of those of $\pi'$.

\subsection{A reduction step}
In this paragraph, we reduce the computation to a particular case that will be treated at the end of the section.
\par

We write 
$\pi = L(\Delta_{\rho_1}[x_1,y_1], \dots, \Delta_{\rho_r}[x_r,y_r], \Delta_\rho[0,-1]^t; \pi(\phi, \eta))$
as a Langlands subrepresentation, 
where 
\begin{itemize}
\item
$\phi \in \Phi_\temp(G)$; 
\item
$t \geq 0$; 
\item
$x_1+y_1 \leq \dots \leq x_r+y_r < 0$; 
\item
$\Delta_{\rho_i}[x_i,y_i] \not\cong \Delta_\rho[0,-1]$ for $i=1, \dots,r$. 
\end{itemize}
We know by the assumption (b) that $x_i \geq 0$ if $\rho_i \cong \rho$. 
Also, by the last condition above, we have $y_i \not= -1$ if $\rho_i \cong \rho$. 
Set $\pi_A \coloneqq L(\Delta_\rho[0,-1]^t; \pi(\phi, \eta))$. 
\par

To rephrase the assumption (a), we recall Jantzen's algorithm (\cite[\S 3.3]{J-dual}). 
Let $\pi_A' \coloneqq D_{\rho|\cdot|^1}^{(l)}(\pi_A)$ be the highest $\rho|\cdot|^1$-derivative of $\pi_A$. 
It can be computed thanks to Theorem \ref{der-special} and Proposition \ref{bad-special}. 
Then Jantzen's Claim 2 in \cite[\S 3.3]{J-dual} says that
\[
\pi \hookrightarrow 
L(\Delta_{\rho_1}[x_1,y_1], \dots, \Delta_{\rho_r}[x_r,y_r], (\rho|\cdot|^1)^{l}) 
\rtimes \pi_A'. 
\]
According to his algorithm, $\pi$ is $\rho|\cdot|^1$-reduced if and only if 
$L(\Delta_{\rho_1}[x_1,y_1], \dots, \Delta_{\rho_r}[x_r,y_r], (\rho|\cdot|^1)^{l})$ is left $\rho|\cdot|^1$-reduced. 
For $i = r+1, \dots, r+l$, we set $\Delta_{\rho_i}[x_i,y_i] = \rho|\cdot|^{1}$. 
Define
\begin{align*}
A_\rho &\coloneqq \{i \in \{1, \dots, r+l\} \;|\; \rho_i \cong \rho,\; x_i = 0\}, \\
A_{\rho|\cdot|^1} &\coloneqq \{i \in \{1, \dots, r+l\} \;|\; \rho_i \cong \rho,\; x_i = 1\}. 
\end{align*}
As in \S \ref{negative}, 
we regard these sets as totally ordered sets, 
and we define a traversable relation $\rel$ between $A_{\rho|\cdot|^1}$ and $A_\rho$.
Let $f \colon A_{\rho}^0 \rightarrow A_{\rho|\cdot|^1}^0$ be the best matching function. 
Then by \cite[Theorem 5.11]{LM}, 
$L(\Delta_{\rho_1}[x_1,y_1], \dots, \Delta_{\rho_r}[x_r,y_r], (\rho|\cdot|^1)^{l})$ is left $\rho|\cdot|^1$-reduced
if and only if $A_{\rho|\cdot|^{1}}^c = \emptyset$. 
Let $D_{[0,1]}^{(k_A)}(\pi_A')$ be the highest $[0,1]$-derivative of $\pi_A'$. 
We will explicitly compute it in Propositions \ref{bad-der-01} and \ref{der-01-special} below. 

\begin{thm}\label{der-01}
Let $\pi\in\Irr(G_n)$ of good or $\rho$-bad parity satisfying the assumptions (a) and (b). 
We use the above notation.
Then the highest $[0,1]$-derivative $D_{[0,1]}^{(k)}(\pi)$ is the unique irreducible subrepresentation of 
\[
L(\Delta_{\rho_1}[x'_1,y_1], \dots, \Delta_{\rho_r}[x'_r,y_r]) \rtimes D_{[0,1]}^{(k_A)}(\pi_A'), 
\]
where 
\[
x'_i = \left\{
\begin{aligned}
&-1	\iif i \in A_\rho^0, \\
&0	\iif i \in A_{\rho|\cdot|^1}, \\
&x_i	\other.
\end{aligned}
\right. 
\]
In particular, $k = k_A + r_1$ with $r_1 \coloneqq |A_{\rho|\cdot|^1}| = |A_\rho^0|$.
\end{thm}
\begin{proof}
Since $x_i \geq 0$ if $\rho_i \cong \rho$, we see that 
$\Delta_{\rho_i}[x_i,y_i] \times Z_{\rho}[0,1] \cong Z_\rho[0,1] \times \Delta_{\rho_i}[x_i,y_i]$
for all $i = 1, \dots, r+l$. 
Hence 
\begin{align*}
\pi 
&\hookrightarrow 
L(\Delta_{\rho_1}[x_1,y_1], \dots, \Delta_{\rho_r}[x_r,y_r], (\rho|\cdot|^1)^{l}) \rtimes \pi_A'
\\&\hookrightarrow 
L(\Delta_{\rho_1}[x_1,y_1], \dots, \Delta_{\rho_r}[x_r,y_r], (\rho|\cdot|^1)^{l}) 
\times Z_\rho[0,1]^{k_A} \rtimes D_{[0,1]}^{(k_A)}(\pi_A')
\\&\cong 
Z_\rho[0,1]^{k_A} \times 
L(\Delta_{\rho_1}[x_1,y_1], \dots, \Delta_{\rho_r}[x_r,y_r], (\rho|\cdot|^1)^{l}) 
\rtimes D_{[0,1]}^{(k_A)}(\pi_A'). 
\end{align*}
\par

We claim that
\[
L(\Delta_{\rho_1}[x_1,y_1], \dots, \Delta_{\rho_r}[x_r,y_r], (\rho|\cdot|^1)^{l}) 
\hookrightarrow 
Z_\rho[0,1]^{r_1} \times L(\Delta_{\rho_1}[x'_1,y_1], \dots, \Delta_{\rho_r}[x'_r,y_r]). 
\]
To see this, by \cite[Proposition 5.6]{LM}, 
it is enough to show that
\begin{align*}
&L(\Delta_{\rho_1}[x_1,y_1], \dots, \Delta_{\rho_r}[x_r,y_r], (\rho|\cdot|^1)^{l}) 
\\&= 
\soc\left( 
\rho^{r_1+k'} \times
\soc \left(
(\rho|\cdot|^1)^{r_1} \times 
L_\rho^{(k')} (L(\Delta_{\rho_1}[x'_1,y_1], \dots, \Delta_{\rho_r}[x'_r,y_r]) ) 
\right)
\right),
\end{align*}
where $L_\rho^{(k')} (L(\Delta_{\rho_1}[x'_1,y_1], \dots, \Delta_{\rho_r}[x'_r,y_r]) )$
is the highest left $\rho$-derivative. 
By our assumptions and by the definition of $x_i'$, 
we see that $k' = r_0-r_1$ with $r_0 = |A_\rho|$ and that 
\[
L_\rho^{(r_0-r_1)} (L(\Delta_{\rho_1}[x'_1,y_1], \dots, \Delta_{\rho_r}[x'_r,y_r]) ) 
= 
L(\Delta_{\rho_1}[x_1^{(1)},y_1], \dots, \Delta_{\rho_r}[x_r^{(1)},y_r])
\]
with 
\begin{align*}
x_i^{(1)} &= \left\{
\begin{aligned}
&-1 \iif i \in A_{\rho}^c, \\
&x_i' \other.
\end{aligned}
\right. 
\\&= 
\left\{
\begin{aligned}
&-1	\iif i \in A_\rho, \\
&0	\iif i \in A_{\rho|\cdot|^1}, \\
&x_i	\other.
\end{aligned}
\right. 
\end{align*}
Since $x_i^{(1)} \not= 1$ if $\rho_i \cong \rho$, 
we have
\begin{align*}
&\soc \left(
(\rho|\cdot|^1)^{r_1} \times 
L_\rho^{(r_0-r_1)} (L(\Delta_{\rho_1}[x'_1,y_1], \dots, \Delta_{\rho_r}[x'_r,y_r]) ) 
\right)
\\&= 
L(\Delta_{\rho_1}[x_1^{(2)},y_1], \dots, \Delta_{\rho_{r+l}}[x_{r+l}^{(2)},y_{r+l}])
\end{align*}
with 
\begin{align*}
x_i^{(2)} &= \left\{
\begin{aligned}
&-1	\iif i \in A_\rho, \\
&1	\iif i \in A_{\rho|\cdot|^1}, \\
&x_i	\other.
\end{aligned}
\right. 
\end{align*}
In particular, we note that $\Delta_{\rho_i}[x_i^{(2)},y_i] \cong \rho|\cdot|^1$ for $i > r$.
Since $x_i^{(2)} \not= 0$ if $\rho_i \cong \rho$, we have
\[
\soc(\rho^{r_0} \rtimes L(\Delta_{\rho_1}[x_1^{(2)},y_1], \dots, \Delta_{\rho_{r+l}}[x_{r+l}^{(2)},y_{r+l}]))
= 
L(\Delta_{\rho_1}[x_1,y_1], \dots, \Delta_{\rho_{r+l}}[x_{r+l},y_{r+l}]). 
\]
Hence we obtain the claim. 
\par

By the claim, we have  
\[
\pi \hookrightarrow Z_{\rho}[0,1]^{k_A+r_1} 
\times 
L(\Delta_{\rho_1}[x'_1,y_1], \dots, \Delta_{\rho_r}[x'_r,y_r])
\rtimes D_{[0,1]}^{(k_A)}(\pi_A'). 
\]
Moreover, by Tadi{\'c}'s formula (Proposition \ref{eq:Tadic}) together with the facts that
\begin{itemize}
\item
$L(\Delta_{\rho_1}[x'_1,y_1], \dots, \Delta_{\rho_r}[x'_r,y_r])$ 
is left $\rho|\cdot|^1$-reduced; 
\item
$L(\Delta_{\rho_1}[x'_1,y_1], \dots, \Delta_{\rho_r}[x'_r,y_r])$ 
is right $\rho$-reduced and right $\rho|\cdot|^{-1}$-reduced; 
\item
$D_{[0,1]}^{(k_A)}(\pi_A')$ is $Z_\rho[0,1]$-reduced and $\rho|\cdot|^1$-reduced, 
\end{itemize}
we see that $L(\Delta_{\rho_1}[x'_1,y_1], \dots, \Delta_{\rho_r}[x'_r,y_r]) \rtimes D_{[0,1]}^{(k_A)}(\pi_A')$ 
is $Z_\rho[0,1]$-reduced and $\rho|\cdot|^1$-reduced. 
Therefore, $D_{[0,1]}^{(k_A+r_1)}(\pi)$ is the highest $[0,1]$-derivative, and 
\[
D_{[0,1]}^{(k_A+r_1)}(\pi) \hookrightarrow 
L(\Delta_{\rho_1}[x'_1,y_1], \dots, \Delta_{\rho_r}[x'_r,y_r])
\rtimes D_{[0,1]}^{(k_A)}(\pi_A'). 
\]
Since this induced representation in the right hand side is a subrepresentation of a standard module, it is SI. 
In particular, $D_{[0,1]}^{(k_A+r_1)}(\pi)$ is 
the unique irreducible subrepresentation of this induced representation. 
\end{proof}

We give now the converse of Theorem \ref{der-01}.
Namely, when $\pi$ is of good or $\rho$-bad parity satisfying the assumptions (a) and (b), 
we will recover the Langlands data of $\pi$ from those of $D_{[0,1]}^{(k)}(\pi)$. 
\par

Write 
$D_{[0,1]}^{(k)}(\pi) = 
L(\Delta_{\rho_1}[x'_1,y_1], \dots, \Delta_{\rho_r}[x'_r,y_r], (\rho|\cdot|^{-1})^s, \Delta_\rho[0,-1]^t; \pi(\phi', \eta'))
$
as a Langlands subrepresentation, where 
\begin{itemize}
\item
$\phi' \in \Phi_\temp(G)$; 
\item
$s, t \geq 0$;
\item
$x'_1+y_1 \leq \dots \leq x'_r+y_r < 0$; 
\item
$\Delta_{\rho_i}[x'_i,y_i] \not\cong \rho|\cdot|^{-1}, \Delta_\rho[0,-1]$ for $i=1, \dots,r$. 
\end{itemize}
Set $\pi''_A \coloneqq L((\rho|\cdot|^{-1})^s, \Delta_\rho[0,-1]^t; \pi(\phi', \eta'))$. 
Define
\begin{align*}
B_{\rho|\cdot|^{-1}} &\coloneqq \{i \in \{1, \dots, r\} \;|\; \rho_i \cong \rho,\; x'_i = -1\}, \\
B_{\rho} &\coloneqq \{i \in \{1, \dots, r\} \;|\; \rho_i \cong \rho,\; x'_i = 0\}
\end{align*}
with the best matching function $f' \colon B_{\rho|\cdot|^{-1}}^0 \rightarrow B_\rho^0$. 
By Theorem \ref{der-01}, we see that $x'_i \not= 1$ if $\rho_i \cong \rho$.
Also, if we set $r_1 \coloneqq |B_{\rho|\cdot|^{-1}}|$, 
$k_A \coloneqq k-r_1$ and $l \coloneqq r_1 - |B_\rho^0|$, 
then we have $k_A \geq 0$ and $l \geq 0$. 

\begin{cor}\label{soc-01}
Let $\pi\in\Irr(G_n)$ of good or $\rho$-bad parity satisfying the assumptions (a) and (b). 
Then $\pi$ is the unique irreducible subrepresentation of 
\[
L(\Delta_{\rho_1}[x_1, y_1], \dots, \Delta_{\rho_r}[x_r, y_r]) \rtimes \pi_A,
\]
where
\[
x_i = \left\{
\begin{aligned}
&0 \iif i \in B_{\rho|\cdot|^{-1}}, \\
&1 \iif i \in B_\rho^0, \\
&x_i' \other, 
\end{aligned}
\right. 
\]
and 
\[
\pi_A \coloneqq S_{\rho|\cdot|^1}^{(l)} \circ S_{[0,1]}^{(k_A)}(\pi_A'').
\]
\end{cor}
\begin{proof}
This follows from Theorem \ref{der-01}.
\end{proof}

\subsection{The representation $\pi_A$ in the bad parity case}\label{sub:bad01}
We keep notation as in the previous paragraph. 
We are left to give an explicit formula for the highest $[0,1]$-derivative of $\pi'_A$ 
and to show how to recover the Langlands data of $\pi'_A$ from those of its highest $[0,1]$-derivative.
\par

We treat first the bad parity case, which is much simpler. 
Recall that $\pi_A = L(\Delta_\rho[0,-1]^t; \pi(\phi, \eta))$ with $\phi \in \Phi_\temp(G)$. 
Let $\pi_A' \coloneqq D_{\rho|\cdot|^1}^{(l)}(\pi_A)$ to be the highest $\rho|\cdot|^1$-derivative of $\pi_A$. 
By Proposition \ref{bad-special}, 
$\pi'_A = L(\Delta_\rho[0,-1]^{t-\kappa}; \pi(\phi', \eta'))$ 
with $\kappa \in \{0,1\}$ with $t \equiv \kappa \bmod 2$
and $\phi' \in \Phi_\temp(G)$ which does not contain $\rho \boxtimes S_3$.
In particular, $t-\kappa$ is even.
Hence what we have to prove is the following.

\begin{prop}\label{bad-der-01}
Let $\pi = L(\Delta_\rho[0,-1]^{t}; \pi(\phi, \eta))$ be of the $\rho$-bad parity 
with $t$ even and 
$\phi \in \Phi_\temp(G)$ such that $\phi \not\supset \rho \boxtimes S_3$. 
Then the highest $[0,1]$-derivative of $\pi$ is 
\[
D_{[0,1]}^{(t)}(\pi) = \pi(\phi, \eta). 
\]
\end{prop}
\begin{proof}
Write $m \coloneqq m_\phi(\rho)$, which is even.
Since 
\begin{align*}
\pi 
&\hookrightarrow \rho^{t+\half{m}} \rtimes L((\rho|\cdot|^{-1})^t; \pi(\phi-\rho^m,\eta))
\\&\cong \rho^{t+\half{m}} \times (\rho|\cdot|^{-1})^t \rtimes \pi(\phi-\rho^m,\eta)
\\&\cong \rho^{t+\half{m}} \times (\rho|\cdot|^{1})^t \rtimes \pi(\phi-\rho^m,\eta), 
\end{align*}
we see that $D_{[0,1]}^{(t)}(\pi)$ is the highest $[0,1]$-derivative and 
\[
D_{[0,1]}^{(t)}(\pi) \hookrightarrow \rho^{\half{m}} \rtimes \pi(\phi-\rho^m,\eta) = \pi(\phi, \eta). 
\]
Since the right hand side is irreducible, this inclusion is an isomorphism.
\end{proof}
By this proposition, it is easy to recover $\pi$ from its highest $[0,1]$-derivative. 

\subsection{The representation $\pi_A$  in the good parity case}\label{subpos01}
To finish our algorithm we need to consider the case of $\pi = L(\Delta_\rho[0,-1]^t; \pi(\phi, \eta))$ 
with $\phi \in \Phi_\gp(G)$ and $\eta \in \widehat{\Sc_\phi}$, 
and $\rho$ is self-dual of the same type as $\phi$.
Furthermore we assume that $\pi$ is $\rho|\cdot|^1$-reduced, which is equivalent that 
if $\rho \boxtimes S_3 \subset \phi$, then 
$m_\phi(\rho) > 0$, $m_\phi(\rho \boxtimes S_3) = 1$ and $\eta(\rho)\eta(\rho \boxtimes S_3) \not= (-1)^t$. 
We determine the highest $[0,1]$-derivative of $\pi$. 

\begin{prop}\label{der-01-special}
Let $\pi = L(\Delta_\rho[0,-1]^t; \pi(\phi, \eta))$ with
$\phi \in \Phi_\gp(G)$ and $\eta \in \widehat{\Sc_\phi}$. 
Suppose that $\rho$ is self-dual of the same type as $\phi$, 
and that $\pi$ is $\rho|\cdot|^1$-reduced. 
Write $m \coloneqq m_\phi(\rho)$. 

\begin{enumerate}
\item
If $\rho \boxtimes S_3 \subset \phi$ and $m$ is odd, 
then the highest $[0,1]$-derivative of $\pi$ is 
\[
D_{[0,1]}^{(t)}(\pi) = \left\{
\begin{aligned}
&\pi(\phi, \eta) \iif t \equiv 0 \bmod 2, \\
&L(\rho|\cdot|^{-1}; \pi(\phi+\rho-\rho \boxtimes S_3, \eta)) \iif t \equiv 1 \bmod 2.
\end{aligned}
\right. 
\]

\item
If $\rho \boxtimes S_3 \subset \phi$ and $m$ is even, 
then the highest $[0,1]$-derivative of $\pi$ is 
\[
D_{[0,1]}^{(t+1)}(\pi) = \pi(\phi - \rho \boxtimes (S_1+S_3), \eta_{t+1}).  
\]

\item
If $\rho \boxtimes S_3 \not\subset \phi$ and $m$ is odd, 
then the highest $[0,1]$-derivative of $\pi$ is 
\[
\left\{
\begin{aligned}
D_{[0,1]}^{(0)}(\pi) &= \pi(\phi, \eta) \iif t = 0, \\
D_{[0,1]}^{(t-1)}(\pi) &= L(\rho|\cdot|^{-1}; \pi(\phi + \rho^2, \eta)) \iif t > 0, t \equiv 0 \bmod 2,\\
D_{[0,1]}^{(t-1)}(\pi) &= L(\Delta_\rho[0,-1]; \pi(\phi, \eta)) \iif t > 0, t \equiv 1 \bmod 2.
\end{aligned}
\right. 
\]

\item
If $\rho \boxtimes S_3 \not\subset \phi$ and $m$ is even, 
then the highest $[0,1]$-derivative of $\pi$ is 
\[
D_{[0,1]}^{(t)}(\pi) = \pi(\phi, \eta_t).
\]
\end{enumerate}
Here, in (2) and (4), we set
\[
\eta_t(\rho' \boxtimes S_d) = 
\left\{
\begin{aligned}
&(-1)^t \eta(\rho) \iif \rho' \boxtimes S_d \cong \rho, \\ 
&\eta(\rho' \boxtimes S_d) \other. 
\end{aligned}
\right. 
\]
\end{prop}
\begin{proof}
We note that $\pi \hookrightarrow \rho^{t+u} \times L((\rho|\cdot|^{-1})^{t}; \pi(\phi - \rho^{2u}, \eta))$
in all cases, where $m = 2u+1$ or $m = 2u$. 
We will apply Theorem \ref{soc-positive} to $L((\rho|\cdot|^{-1})^{t}; \pi(\phi - \rho^{2u}, \eta))$ and $x=1$ in each case.
\par

We show (1). 
Write $m = 2u+1$. 
By Theorem \ref{soc-positive}, we have
\[
\pi \hookrightarrow 
\rho^{t+u} \times (\rho|\cdot|^{1})^{t} \rtimes 
\left\{
\begin{aligned}
&\pi(\phi-\rho^{2u}, \eta) \iif t \equiv 0 \bmod 2, \\
&L(\rho|\cdot|^{-1}; \pi(\phi - \rho^{2u-1} - \rho \boxtimes S_3, \eta)) \iif t \equiv 1 \bmod 2. 
\end{aligned}
\right. 
\]
Note that $\rho^u \rtimes \pi(\phi-\rho^{2u}, \eta) = \pi(\phi, \eta)$
and $\rho^u \rtimes L(\rho|\cdot|^{-1}; \pi(\phi - \rho^{2u-1} - \rho \boxtimes S_3, \eta)) 
= L(\rho|\cdot|^{-1}; \pi(\phi + \rho - \rho \boxtimes S_3, \eta))$
are both irreducible by \cite[Proposition 2.4.3]{Ar} and M{\oe}glin's construction (see \cite[\S 8]{X2}). 
Hence
\begin{align*}
\pi \hookrightarrow Z_\rho[0,1]^t \rtimes 
\left\{
\begin{aligned}
&\pi(\phi, \eta) \iif t \equiv 0 \bmod 2, \\
&L(\rho|\cdot|^{-1}; \pi(\phi + \rho - \rho \boxtimes S_3, \eta)) \iif t \equiv 1 \bmod 2. 
\end{aligned}
\right. 
\end{align*}
This shows (1). 
\par

We show (2). 
Write $m = 2u$.
Note that $u > 0$ and $\eta(\rho \boxtimes S_3) = (-1)^{t+1}\eta(\rho)$. 
Hence
\[
\pi \hookrightarrow 
\rho^{t+u} \times (\rho|\cdot|^1)^{t+1} \rtimes \pi(\phi-\rho^{2u-1}-\rho \boxtimes S_3, \eta_{t+1}). 
\]
This implies that
\begin{align*}
\pi &\hookrightarrow 
Z_\rho[0,1]^{t+1} \times \rho^{u-1} \rtimes \pi(\phi-\rho^{2u-1}-\rho \boxtimes S_3, \eta_{t+1}) 
\\&= Z_\rho[0,1]^{t+1} \rtimes \pi(\phi-\rho-\rho \boxtimes S_3, \eta_{t+1}).
\end{align*}
This shows (2). 
\par

We show (3). 
When $t=0$, it is clear that $\pi$ is $Z_\rho[0,1]$-reduced (Lemma \ref{[01]-red}). 
Suppose that $t > 0$. 
Write $m = 2u+1$. 
Since 
\[
\pi 
\hookrightarrow 
\rho^{t+u} \times (\rho|\cdot|^1)^{t-1} \rtimes L(\rho|\cdot|^{-1}; \pi(\phi - \rho^{2u}, \eta)), 
\]
we have 
\[
\pi \hookrightarrow Z_\rho[0,1]^{t-1} \times \rho^{u+1} \rtimes L(\rho|\cdot|^{-1}; \pi(\phi - \rho^{2u}, \eta)). 
\]
By \cite[Proposition 2.4.3]{Ar} and M{\oe}glin's construction (see \cite[\S 8]{X2}), we have
\[
\rho^{u+1} \rtimes L(\rho|\cdot|^{-1}; \pi(\phi - \rho^{2u}, \eta))
= 
L(\rho|\cdot|^{-1}; \pi(\phi + \rho^{2}, \eta)) \oplus L(\Delta_\rho[0,-1]; \pi(\phi, \eta)).
\]
In particular, $D_{[0,1]}^{(t-1)}(\pi)$ is the highest $[0,1]$-derivative, 
and is isomorphic to one of the two direct summands in the right hand side. 
Now we note that $L(\Delta_\rho[0,-1], \Delta_\rho[1,0]) \cong \soc(Z_\rho[0,1] \times Z_\rho[-1,0])$. 
When $t$ is odd, by \cite[Proposition 2.4.3]{Ar}, we have 
\begin{align*}
\pi &\hookrightarrow L(\Delta_\rho[0,-1], \Delta_\rho[1,0])^{\half{t-1}} \rtimes L(\Delta_\rho[0,-1]; \pi(\phi, \eta)). 
\end{align*}
Since $L(\Delta_\rho[0,-1]; \pi(\phi, \eta))$ is $\rho|\cdot|^1$-reduced and $Z_\rho[0,1]$-reduced, 
by considering Tadi{\'c}'s formula (Proposition \ref{eq:Tadic}), we see that 
\[
D_{[0,1]}^{(t-1)}\left( 
L(\Delta_\rho[0,-1], \Delta_\rho[1,0])^{\half{t-1}} \rtimes L(\Delta_\rho[0,-1]; \pi(\phi, \eta)) 
\right)
= L(\Delta_\rho[0,-1]; \pi(\phi, \eta)), 
\]
which implies that $D_{[0,1]}^{(t-1)}(\pi) = L(\Delta_\rho[0,-1]; \pi(\phi, \eta))$. 
When $t = 2$, by \cite[Proposition 2.4.3]{Ar}, we have 
\begin{align*}
\pi &\hookrightarrow L(\Delta_\rho[0,-1], \Delta_\rho[1,0]) \rtimes \pi(\phi, \eta)
\\& \cong \soc(Z_\rho[0,1] \times Z_\rho[-1,0]) \rtimes \pi(\phi, \eta)
\\&\hookrightarrow Z_\rho[0,1] \times \rho|\cdot|^{-1} \rtimes \pi(\phi + \rho^2, \eta), 
\end{align*}
which implies that $D_{[0,1]}^{(1)}(\pi) = L(\rho|\cdot|^{-1}; \pi(\phi + \rho^{2}, \eta))$.
When $t > 2$ is even, we have 
\begin{align*}
\pi &\hookrightarrow L(\Delta_\rho[0,-1], \Delta_\rho[1,0])^{\half{t-2}} \rtimes L(\Delta_\rho[0,-1]^2; \pi(\phi, \eta))
\\&\hookrightarrow 
Z_\rho[0,1] \times L(\Delta_\rho[0,-1], \Delta_\rho[1,0])^{\half{t-2}} 
\rtimes L(\rho|\cdot|^{-1}; \pi(\phi + \rho^{2}, \eta)). 
\end{align*}
Since $L(\rho|\cdot|^{-1}; \pi(\phi + \rho^{2}, \eta))$ is $\rho|\cdot|^1$-reduced and $Z_\rho[0,1]$-reduced, 
by considering Tadi{\'c}'s formula (Proposition \ref{eq:Tadic}), we see that 
\[
D_{[0,1]}^{(t-1)}\left( 
Z_\rho[0,1] \times L(\Delta_\rho[0,-1], \Delta_\rho[1,0])^{\half{t-2}} 
\rtimes L(\rho|\cdot|^{-1}; \pi(\phi + \rho^{2}, \eta))
\right)
= L(\rho|\cdot|^{-1}; \pi(\phi + \rho^{2}, \eta)), 
\]
which implies that $D_{[0,1]}^{(t-1)}(\pi) = L(\rho|\cdot|^{-1}; \pi(\phi + \rho^{2}, \eta))$.
We obtain (3). 
\par

We show (4). 
Write $m = 2u$. 
Since 
\[
\pi \hookrightarrow \rho^{t+u} \times (\rho|\cdot|^1)^t \rtimes \pi(\phi- \rho^{2u}, \eta), 
\]
we have
\[
\pi \hookrightarrow Z_\rho[0,1]^{t} \times \rho^{u} \rtimes \pi(\phi - \rho^{2u}, \eta). 
\]
In particular, this shows (4) when $u=0$. 
Hereafter we assume that $u > 0$. 
Then 
\[
\rho^u \rtimes \pi(\phi - \rho^{2u}, \eta) = \pi(\phi, \eta_t) \oplus \pi(\phi, \eta_{t+1}).
\] 
To show $\pi \hookrightarrow Z_\rho[0,1]^{t} \rtimes \pi(\phi, \eta_t)$,
we use an argument inspired by M{\oe}glin's construction of $A$-packets. 
\par

Write $\phi = \rho^{m} \oplus (\oplus_{i=1}^r \rho_i \boxtimes S_{d_i})$
with $d_1 \leq \dots \leq d_r$ and $d_i > 3$ if $\rho_i \cong \rho$.
Choose $\phi_> = (\oplus_{j=1}^{m} \rho \boxtimes S_{2x_j+1}) \oplus (\oplus_{i=1}^r \rho_i \boxtimes S_{d_i'})$ 
such that
\begin{itemize}
\item
$x_j \in \Z$ with $x_j > 1$; 
\item
$d_i' \equiv d_i \bmod 2$ with $d_i' \geq d_i$;
\item
$2x_1+1 < \dots < 2x_{m}+1 < d_1' < \dots < d_r'$. 
\end{itemize}
Define $\eta_> \in \widehat{\Sc_{\phi_>}}$ by 
$\eta_>(\rho \boxtimes S_{2x_j+1}) = (-1)^t\eta(\rho)$ 
and $\eta_>(\rho_i \boxtimes S_{d_i'}) = \eta(\rho_i \boxtimes S_{d_i})$.
Then 
$\pi(\phi,\eta_t) = J_2 \circ J_1(\pi(\phi_>, \eta_>))$
with 
\begin{align*}
J_1 &= \Jac_{\rho|\cdot|^{x_{m}}, \dots, \rho|\cdot|^1} \circ \dots \circ \Jac_{\rho|\cdot|^{x_{1}}, \dots, \rho|\cdot|^1}, \\
J_2 &= 
\Jac_{\rho_t|\cdot|^{\half{d_r'-1}}, \dots, \rho_t|\cdot|^{\half{d_r+1}}}
\circ \dots \circ
\Jac_{\rho_1|\cdot|^{\half{d_1'-1}}, \dots, \rho_1|\cdot|^{\half{d_1+1}}}, 
\end{align*}
where we set 
$\Jac_{\rho|\cdot|^{x}, \dots, \rho|\cdot|^y} = D_{\rho|\cdot|^y}^{(1)} \circ \dots \circ D_{\rho|\cdot|^x}^{(1)}$.
Since $\phi_>$ contains neither $\rho$ nor $\rho \boxtimes S_3$, 
by the argument in the previous paragraph, 
we have
\[
\soc(Z_\rho[0,1]^t \rtimes \pi(\phi_>, \eta_>)) = L(\Delta_\rho[0,-1]^t; \pi(\phi_>, \eta_>)).
\]
By Theorem \ref{soc-positive}, using the assumption that $m \equiv 0 \bmod 2$, we see that 
\[
J_2 \circ J_1(L(\Delta_\rho[0,-1]^t; \pi(\phi_>, \eta_>)))
= L(\Delta_\rho[0,-1]^t; \pi(\phi, \eta)) = \pi.
\]
On the other hand, since 
\[
\pi(\phi_>,\eta_>) \hookrightarrow 
\Delta_{\rho}[x_1,1] \times \dots \times \Delta_\rho[x_{m},1] \rtimes J_1(\pi(\phi_>,\eta_>))
\]
by \cite[Lemma 5.7]{X1}, 
and since $Z_\rho[0,1] \times \Delta_\rho[x,1] \cong \Delta_\rho[x,1] \times Z_\rho[0,1]$  if $x \geq 1$, 
we see that
\begin{align*}
&J_2 \circ J_1 (\soc(Z_\rho[0,1]^t \rtimes \pi(\phi_>, \eta_>)))
\\&\hookrightarrow 
J_2 \circ J_1 (Z_\rho[0,1]^t \rtimes \pi(\phi_>, \eta_>))
\\&\hookrightarrow
J_2 \circ J_1 (\Delta_{\rho}[x_1,1] \times \dots \times \Delta_\rho[x_{m},1] 
\times Z_\rho[0,1]^t \rtimes J_1(\pi(\phi_>, \eta_>)))
\\&=J_2(Z_\rho[0,1]^t \rtimes J_1(\pi(\phi_>, \eta_>))).  
\end{align*}
Finally, since $(d_i+1)/2 > 2$ if $\rho_i \cong \rho$, we have
\[
J_2(Z_\rho[0,1]^t \rtimes J_1(\pi(\phi_>, \eta_>)))
= 
Z_\rho[0,1]^t \rtimes J_2 \circ J_1(\pi(\phi_>, \eta_>))
=
Z_\rho[0,1]^t \rtimes \pi(\phi, \eta_t).
\]
Therefore we conclude that $\pi \hookrightarrow Z_\rho[0,1]^t \rtimes \pi(\phi, \eta_t)$. 
This completes the proof of (4).
\end{proof}

Finally, we state the converse of Proposition \ref{der-01-special} in terms of $A$-parameters. 
\begin{cor}\label{soc-01-special}
Let $\pi = L(\Delta_\rho[0,1]^t; \pi(\phi, \eta))$ be the same as Proposition \ref{der-01-special}, 
and $D_{[0,1]}^{(k)}(\pi)$ be the highest $[0,1]$-derivative of $\pi$. 
Suppose that $k > 0$.
Then one can write $D_{[0,1]}^{(k)}(\pi) = L((\rho|\cdot|^{-1})^{s'}, \Delta_\rho[0,1]^{t'}; \pi(\phi', \eta'))$
with $s'+t'+m_{\phi'}(\rho \boxtimes S_3) \leq 1$. 
Moreover, with $m' \coloneqq m_{\phi'}(\rho)$, we have the following. 
\begin{enumerate}
\item
If $s'=1$, then $m' \geq 2$, $k \equiv 1\bmod 2$ and 
\[
\pi = \pi(\phi' - \rho^2 + (\rho \boxtimes S_2 \boxtimes S_2)^{k+1},m',\eta'). 
\]

\item
If $t'=1$, then $m' \equiv 1 \bmod 2$, $k \equiv 0 \bmod 2$ and
\[
\pi = \pi(\phi' + (\rho \boxtimes S_2 \boxtimes S_2)^{k+1},1,\eta'). 
\]

\item
If $m_{\phi'}(\rho \boxtimes S_3) = 1$, then $m' \equiv 1 \bmod 2$, $k \equiv 0 \bmod 2$ and
\[
\pi = \pi(\phi' + (\rho \boxtimes S_2 \boxtimes S_2)^{k},1,\eta'). 
\]

\item
If $s'+t'+m_{\phi'}(\rho \boxtimes S_3) = 0$, then 
\[
\pi = \pi(\phi' + (\rho \boxtimes S_2 \boxtimes S_2)^{k},m'+1,\eta'_k),  
\]
where $\eta_k'(\rho) = (-1)^{k}\eta'(\rho)$.
\end{enumerate}
\end{cor}
\begin{proof}
This follows from Proposition \ref{der-01-special}.
\end{proof}
\par

\section{Some examples of Zelevinsky--Aubert duality}
By the results in previous sections, we have completed Algorithm \ref{alg} to compute the Zelevinsky--Aubert duality.
In this section, we give some examples. 
Here we set $\rho \coloneqq \1_{\GL_1(F)}$, and we drop $\rho$ from the notation. 
For example, we write $\Delta[x,y] \coloneqq \Delta_\rho[x,y]$ and $Z[y,x] \coloneqq Z_\rho[y,x]$.
When $\phi = \oplus_{i=1}^r S_{d_i} \in \Phi_\gp(G)$ and $\eta(S_{d_i}) = \eta_i \in \{\pm1\}$, 
we write $\pi(\phi, \eta) = \pi(d_1^{\eta_1}, \dots, d_r^{\eta_r})$. 
\par

\subsection{Example 1}
Let us compute the Zelevinsky--Aubert dual of 
\[
L(\Delta[0,-2], \Delta[0,-1]; \pi(3^+)) \in \Irr(\Sp_{10}(F)).
\] 
Note that it is of good parity, and it is $|\cdot|^z$-reduced for $z \not= 0$ by Theorem \ref{soc-positive}.
By Algorithm \ref{alg}, we have the following commutative diagram:
\[
\xymatrix{
L(\Delta[0,-2], \Delta[0,-1]; \pi(3^+))
\ar@{|->}[d]_{D_{\Delta[0,-1]}^{(2)}} 
\ar@{|->}[rr]^{\pi \mapsto \hat\pi} 
&&
L(\Delta[0,-2], \Delta[0,-1]; \pi(3^+))
\\
L(|\cdot|^{-2}; \pi(3^+))
\ar@{|->}[d]_{D_{|\cdot|^{-2}}^{(1)}} \ar@{|->}[rr]^{\pi \mapsto \hat\pi} 
&&
L(\Delta[-1,-2]; \pi(1^+))
\ar@{|->}[u]_{S_{Z[0,1]}^{(2)}} 
\\
\pi(3^+)
\ar@{|->}[d]_{D_{|\cdot|^1}^{(1)}} \ar@{|->}[rr]^{\pi \mapsto \hat\pi} 
&&
L(|\cdot|^{-1}; \pi(1^+))
\ar@{|->}[u]_{S_{|\cdot|^2}^{(1)}} 
\\
\pi(1^+)
\ar@{|->}[rr]^{\pi \mapsto \hat\pi} 
&&
\pi(1^+)
\ar@{|->}[u]_{S_{|\cdot|^{-1}}^{(1)}} 
}
\]
For the computation of $S_{Z[0,1]}^{(2)}$, 
by Corollaries \ref{soc-01}, \ref{soc-01-special} and Theorem \ref{der-special}, we have 
\begin{align*}
S_{Z[0,1]}^{(2)}(L(\Delta[-1,-2]; \pi(1^+)))
&= 
\soc\left( \Delta[0,-2] \rtimes S_{|\cdot|^1}^{(1)} \circ S_{Z[0,1]}^{(1)}(\pi(1^+)) \right)
\\&=
\soc\left( \Delta[0,-2] \rtimes S_{|\cdot|^1}^{(1)}(\pi(1^-,1^-,3^+)) \right)
\\&=
L\left( \Delta[0,-2], \Delta[0,-1]; \pi(3^+) \right).
\end{align*}
In conclusion, we see that 
$L(\Delta[0,-2], \Delta[0,-1]; \pi(3^+))$ is fixed by the Zelevinsky--Aubert duality.

\subsection{Example 2}
Next, let us compute the Zelevinsky--Aubert dual of 
\[
\pi(1^\epsilon,1^\epsilon,3^+,5^-,5^-) \in \Irr_\temp(\Sp_{14}(F))
\]
for $\epsilon \in \{\pm\}$.
First, we compute derivatives: 
\[
\xymatrix@C=0.1pt{
\pi(1^+,1^+,3^+,5^-,5^-) 
\ar@{|->}[d]_{D_{|\cdot|^2}^{(1)}} 
&& 
\pi(1^-,1^-,3^+,5^-,5^-) 
\ar@{|->}[d]^{D_{|\cdot|^2}^{(1)}} 
\\
L(\Delta[1,-2]; \pi(1^+,1^+,3^+))
\ar@{|->}[d]_{D_{|\cdot|^1}^{(2)}} 
&&
L(\Delta[1,-2]; \pi(1^-,1^-,3^+))
\ar@{|->}[d]^{D_{|\cdot|^1}^{(1)}} 
\\
L(\Delta[0,-2]; \pi(1^+,1^+,1^+))
\ar@{|->}[d]_{D_{|\cdot|^2}^{(1)}} 
&&
L(\Delta[0,-2]; \pi(1^-,1^-,3^+))
\ar@{|->}[d]^{D_{\Delta[0,-1]}^{(1)}} 
\\
L(\Delta[0,-1]; \pi(1^+,1^+,1^+))
\ar@{|->}[d]_{D_{\Delta[0,-1]}^{(1)}} 
&&
L(|\cdot|^{-2}; \pi(1^-,1^-,3^+))
\ar@{|->}[d]^{D_{|\cdot|^{-2}}^{(1)}} 
\\
\pi(1^+,1^+,1^+),
&&
\pi(1^-,1^-,3^+).
}
\]
By Proposition \ref{temp}, 
we have $\hat\pi(1^+,1^+,1^+) = \pi(1^+,1^+,1^+)$ and $\hat\pi(1^-,1^-,3^+) = L(\Delta[0,-1]; \pi(1^+))$.
Next we compute socles: 
\[
\xymatrix@C=0.1pt{
\pi(1^+,1^+,1^+)
\ar@{|->}[d]_{S_{Z[0,1]}^{(1)}} 
&&
L(\Delta[0,-1]; \pi(1^+))
\ar@{|->}[d]^{S_{|\cdot|^2}^{(1)}} 
\\
\pi(1^-,1^-,1^-,1^-,3^+)
\ar@{|->}[d]_{S_{|\cdot|^{-2}}^{(1)}} 
&&
L(\Delta[0,-2]; \pi(1^+))
\ar@{|->}[d]^{S_{Z[0,1]}^{(1)}} 
\\
L(|\cdot|^{-2}; \pi(1^-,1^-,1^-,1^-,3^+))
\ar@{|->}[d]_{S_{|\cdot|^{-1}}^{(2)}} 
&&
L(\Delta[0,-2]; \pi(1^-,1^-,3^+))
\ar@{|->}[d]^{S_{|\cdot|^{-1}}^{(1)}} 
\\
L(\Delta[-1,-2], |\cdot|^{-1}; \pi(1^-,1^-,1^-,1^-,3^+))
\ar@{|->}[d]_{S_{|\cdot|^{-2}}^{(1)}} 
&&
L(\Delta[0,-2], |\cdot|^{-1}; \pi(1^-,1^-,3^+))
\ar@{|->}[d]^{S_{|\cdot|^{-2}}^{(1)}}
\\
L(\Delta[-1,-2], |\cdot|^{-2}, |\cdot|^{-1}; \pi(1^-,1^-,1^-,1^-,3^+)),
&& 
L(\Delta[0,-2], |\cdot|^{-2}, |\cdot|^{-1}; \pi(1^-,1^-,3^+)).
}
\]
Therefore, we conclude that
\begin{align*}
\hat\pi(1^+,1^+,3^+,5^-,5^-) &= L(\Delta[-1,-2], |\cdot|^{-2}, |\cdot|^{-1}; \pi(1^-,1^-,1^-,1^-,3^+)), \\
\hat\pi(1^-,1^-,3^+,5^-,5^-) &= L(\Delta[0,-2], |\cdot|^{-2}, |\cdot|^{-1}; \pi(1^-,1^-,3^+)). 
\end{align*}
Similarly, 
one can prove that $\hat\pi(3^+,5^-,5^-) = L(\Delta[-1,-2], |\cdot|^{-2}, |\cdot|^{-1} ; \pi(1^-,1^-,3^+))$. 
Hence we see that 
\begin{align*}
&\1_{\GL_1(F)} \rtimes L(\Delta[-1,-2], |\cdot|^{-2}, |\cdot|^{-1} ; \pi(1^-,1^-,3^+))
\\\cong&
L(\Delta[-1,-2], |\cdot|^{-2}, |\cdot|^{-1}; \pi(1^-,1^-,1^-,1^-,3^+))
\\&\oplus 
L(\Delta[0,-2], |\cdot|^{-2}, |\cdot|^{-1}; \pi(1^-,1^-,3^+)).
\end{align*}
In these computations, 
we also proved, for example, that $L(\Delta[0,-2]; \pi(1^-,1^-,3^+))$ is fixed by the Zelevinsky--Aubert duality.
This fact does not follow from results in \cite{At2}.
As in this example, 
even if $\pi$ is tempered, we need to compute $S_{Z[0,1]}^{(k)}$ in general. 
\par



\begin{thebibliography}{20}

\bibitem{A1}
{D. Alvis}, 
{\em The duality operation in the  character ring of a finite Chevalley group}. 
\emph{Bull.  Amer. Math. Soc. (N.S.)} {\bf1} (1979), no.~6, 907--911. 

\bibitem{A2}
{D. Alvis}, 
{\em Duality and character values of finite groups of Lie type}. 
\emph{J. Algebra} {\bf74} (1982), no.~1, 211--222.
    
\bibitem{Ar}
{J. Arthur}, 
{\em The endoscopic classification of representations. Orthogonal and symplectic groups}. 
\emph{American Mathematical Society Colloquium Publications,} {\bf61}. 
{\it American Mathematical Society, Providence, RI,} 2013. xviii+590 pp.

\bibitem{At1}
{H. Atobe}, 
{\em Jacquet modules and local Langlands correspondence}. 
\emph{Invent. Math.} {\bf219} (2020), no.~3, 831--871.

\bibitem{At2}
{H. Atobe}, 
{\em On an algorithm to compute derivatives}. 
arXiv:2006.02638v1.

\bibitem{Au}
{A.-M. Aubert}, 
{\em Dualit{\'e} dans le groupe de Grothendieck de la cat{\'e}gorie 
des repr{\'e}sentations lisses de longueur finie d'un groupe r{\'e}ductif $p$-adique}. 
\emph{Trans.~Amer.~Math.~Soc.} {\bf347} (1995), no.~6, 2179--2189
and \emph{Erratum.} ibid. {\bf348} (1996), 4687--4690.

\bibitem{Ber}
{J. Bernstein}, 
{\em Representations of $p$-adic groups}. 
Lectures by Joseph Bernstein. Harvard University, Fall 1992. 
Written by Karl E. Rumelhart. 
Available at: \\
http://people.math.harvard.edu/\textasciitilde gaitsgde/Jerusalem\_2010/GradStudentSeminar/p-adic.pdf

\bibitem{BBK}
{J. Bernstein, R. Bezrukavnikov and D. Kazhdan}, 
{\em Deligne-Lusztig duality and wonderful compactification}. 
\emph{Selecta Math. (N.S.)} {\bf24},  (2018), no.~1, 7--20. 
 
\bibitem{Bez}
{R. Bezrukavnikov}, 
{\em Homological properties of representations 
of $p$-adic groups related to the geometry of the group at infinity}. 
Pr{\'e}publication (2004).

\bibitem{Curtis}
{C. Curtis}, 
{\em Truncation and duality in the character ring of a finite group of Lie type}. 
\emph{J. Algebra} {\bf62} (1980), no.~2, 320--332. 
  
\bibitem{J-dec}
{C. Jantzen}, 
{\em On supports of induced representations for symplectic and odd-orthogonal groups}. 
\emph{Amer. J. Math.} {\bf119} (1997), no.~6, 1213--1262.
  

\bibitem{J-temp}
{C. Jantzen}, 
{\em Tempered representations for classical $p$-adic groups}. 
\emph{Manuscripta Math.} {\bf145} (2014), no.~3-4, 319--387.

\bibitem{J-irrr}
{C. Jantzen}, 
{\em Jacquet modules and irrreducibility of induced representations for classical $p$-adic groups}. 
\emph{Manuscripta Math.} {\bf156} (2018), no.~1-2, 23--55.

\bibitem{J-dual}
{C. Jantzen}, 
{\em Duality for classical $p$-adic groups: the half-integral case}. 
\emph{Represent. Theory} {\bf22} (2018), 160--201.

\bibitem{Kato}
{S.-I. Kato},
{\em Duality for representations of a Hecke algebra}. 
\emph{Proc. Amer. Math. Soc.} {\bf119} (1993), no.~3, 941--946. 

\bibitem{KZ} 
{H. Knight and A. Zelevinsky}, 
{\em Representations of quivers of type A and the multisegment duality}. 
\emph{Adv. Math.} {\bf117} (1996), 273--293.

\bibitem{K}
{T. Konno}, 
{\em A note on the Langlands classification and irreducibility of induced representations of $p$-adic groups}. 
{\it Kyushu J. Math.} {\bf57} (2003), no.~2, 383--409. 

\bibitem{KL}
{A. Kret and E. Lapid}, 
{\em Jacquet modules of ladder representations}. 
\emph{C. R. Math. Acad. Sci. Paris} {\bf350} (2012), no.~21-22, 937--940.

\bibitem{LM}
{E. Lapid and A. M{\'i}nguez}, 
{\em On parabolic induction on inner forms of the general linear group over a non-archimedean local field}.
\emph{Selecta Math. (N.S.)} {\bf22} (2016), no.~4, 2347--2400. 

\bibitem{LT}
{E. Lapid and M. Tadi{\'c}}, 
{\em Some results on reducibility of parabolic induction for classical groups}. 
\emph{Amer. J. Math.} {\bf142} (2020), no.~2, 505--546.

\bibitem{Ma1}
{I. Mati{\'c}}, 
{\em  Aubert duals of strongly positive discrete series and a class of unitarizable representations}. 
\emph{Proc. Amer. Math. Soc.} {\bf145} (2017), no.~8, 3561--3570. 

\bibitem{Ma2}
{I. Mati{\'c}}, 
{\em  Aubert duals of discrete series: the first inductive step}. 
\emph{Glas. Mat. Ser. III} {\bf54(74)} (2019), no.~1, 133--178.


\bibitem{MR} 
{I. Mirkovi{\'c} and S. Riche}, 
{\em Iwahori--Matsumoto involution and linear Koszul duality}. 
\emph{Int. Math. Res. Not.} (2015), no.~1, 150--196.

\bibitem{Moe4}
{C. M{\oe}glin}, 
{\em Sur certains paquets d'Arthur et involution d'Aubert-Schneider-Stuhler g{\'e}n{\'e}ralis{\'e}e}. 
\emph{Represent. Theory}  {\bf10}, (2006), 86--129.

\bibitem{Moe3}
{C. M{\oe}glin}, 
{\em Multiplicit{\'e} $1$ dans les paquets d'Arthur aux places $p$-adiques}. 
\emph{On certain $L$-functions}, 333--374, 
\emph{Clay Math.~Proc.,} {\bf13}, {\it Amer.~Math.~Soc., Providence, RI,} 2011.

\bibitem{MVW} 
{C. M{\oe}glin, M.-F. Vign{\'e}ras and J.-L. Waldspurger}, 
{\em Correspondance de Howe sur un corps $p$-adique}.
\emph{Lecture Notes in Mathematics,} {\bf1291}. 
{\it Springer-Verlag, Berlin,} 1987. viii+163 pp.

\bibitem{MW2}
{C. M{\oe}glin and J.-L. Waldspurger}, 
{\em Sur l'involution de Zelevinski}.
\emph{J. Reine Angew. Math.} {\bf 372} (1986), 136--177. 

\bibitem{MW}
{C. M{\oe}glin and J.-L. Waldspurger}, 
{\em La conjecture locale de Gross-Prasad pour les groupes sp{\'e}ciaux orthogonaux: le cas g{\'e}n{\'e}ral}.
Sur les conjectures de Gross et Prasad. II. 
\emph{Ast{\'e}risque} No.~{\bf347} (2012), 167--216.

\bibitem{SS}
{P. Schneider and U. Stuhler},
{\em Representation theory and sheaves on the Bruhat-Tits building}. 
\emph{Inst. Hautes {\'E}tudes Sci. Publ. Math.} No.~{\bf85} (1997), 97--191.

\bibitem{T}
{M. Tadi{\'c}}, 
{\em Structure arising from induction and Jacquet modules of representations of classical $p$-adic groups}. 
\emph{J. Algebra} {\bf177} (1995), no.~1, 1--33.

\bibitem{T2}
{M. Tadi{\'c}}, 
{\em On unitarizability in the case of classical $p$-adic groups}. 
{\it Geometric aspects of the trace formula,} 405--453, 
\emph{Simons Symp.,} {\it Springer, Cham,} 2018.

\bibitem{T1}
{M. Tadi{\'c}}, 
{\em Unitarizability in Corank Three for Classical $p$-adic Groups}. 
\emph{Memoirs of the American Mathematical Society,} to appear.

\bibitem{X1}
{B. Xu}, 
{\em On the cuspidal support of discrete series for $p$-adic quasisplit $Sp(N)$ and $SO(N)$}. 
\emph{Manuscripta Math.} {\bf154} (2017), no.~3-4, 441--502. 

\bibitem{X2}
{B. Xu}, 
{\em On M{\oe}glin's parametrization of Arthur packets for $p$-adic quasisplit $\Sp(N)$ and $\SO(N)$}. 
\emph{Canad. J. Math.} {\bf69} (2017), no.~4, 890--960.

\bibitem{W}
{J.-L. Waldspurger}, 
{\em Repr{\'e}sentations de r{\'e}duction unipotente pour $SO(2n+1)$, I: Une involution}.
\emph{J. Lie Theory} {\bf 28} (2018), no. 2, 381--426. 

\bibitem{Z}
{A. V. Zelevinsky}, 
{\em Induced representations of reductive $\mathfrak{p}$-adic groups. II. On irreducible representations of $\GL(n)$}.  
\emph{Ann. Sci. {\'E}cole Norm. Sup.} (4) {\bf13} (1980), no.~2, 165--210.

\end{thebibliography}
\end{document}